\documentclass{article}
\usepackage{amsmath, amsthm, amsfonts, amssymb, amscd, stmaryrd,xcolor,enumerate, enumitem}

\usepackage{graphicx}
\usepackage{tikz}

\usepackage[
    backend=biber,
    style=numeric,
    maxbibnames=99
  ]{biblatex}
\renewbibmacro{in:}{%
  \ifentrytype{article}
    {}
    {\bibstring{in}%
     \printunit{\intitlepunct}}}
 \DeclareNameAlias{default}{family-given}

\theoremstyle{definition}

\theoremstyle{plain}
\newtheorem{Thm}{Theorem}
\newtheorem{Lem}[Thm]{Lemma}
\newtheorem{Cor}[Thm]{Corollary}

\newtheorem{Prop}[Thm]{Proposition}

\newtheorem{Example}[Thm]{Example}

\numberwithin{equation}{section}	
\numberwithin{figure}{Prob}
\numberwithin{table}{Prob}
\numberwithin{Thm}{section}
\newcommand{\N}{\mathbb{N}}
\newcommand{\Z}{\mathbb{Z}}

\newcommand{\C}{\mathbb{C}}

\newcommand{\normal}{\trianglelefteq}

\DeclareMathOperator*{\Aut}{Aut}			

\DeclareMathOperator*{\Inn}{Inn}

\renewcommand{\S}{\mathfrak{S}}
\newcommand{\T}{\mathfrak{T}}
\newcommand{\D}{\mathcal{D}}

\newcommand{\Span}{\text{Span}}

\title{Schur rings over cyclic groups having Almost Commutative Terwilliger algebras}
\author{Nicholas L. Bastian\footnote{Email: nicbastian16@gmail.com}, Stephen P. Humphries\footnote{Email: steve@mathematics.byu.edu}}

\date{\today}

\begin{document}

\maketitle

\begin{abstract}
Terwilliger algebras are subalgebras of a matrix algebra constructed from an association scheme. Rie Tanaka defined what it means for a Terwilliger algebra to be almost commutative and gave five equivalent conditions in the case where the association scheme is commutative. A sixth condition for a Terwilliger algebra coming from a commutative Schur ring to be almost commutative has since been discovered. In this paper we first provide a classification of orbit Schur rings that produce an almost commutative Terwilliger algebra for a finite cyclic group. In particular, we show that the subgroup of automorphisms used to form the orbit Schur ring is either trivial, or the whole automorphism group when the cyclic group has prime power order. If the cyclic group has order $2^n$, these are the only options. If the cyclic group has order $p^n$, for an odd prime $p$, then the automorphism subgroup of order $p^{n-1}$ also works. If the group has a non-prime power order then the only orbit Schur ring that produces an almost commutative Terwilliger algebra comes from the trivial subgroup of the automorphism group. We then give a condition for when a wedge product of Schur rings produces an almost commutative Terwilliger algebra. This allows us to determine exactly when a Schur ring over a cyclic group produces an almost commutative Terwilliger algebra.

% This result for orbit Schur rings is then used to classify all Schur rings over cyclic groups that result in an almost commutative Terwilliger algebra.
   \end{abstract}



\section{Introduction}

Terwilliger algebras were originally developed in the 1990's by Paul Terwilliger to study commutative association schemes. In particular, he looked at P-polynomial and Q-polynomial association schemes. Over the course of the three papers \cite{Terwilliger1,Terwilliger2,Terwilliger3} in which Terwilliger algebras were originally introduced, Terwilliger finds a combinatorial characterization of thin P- and Q-polynomial association schemes. Since that time Terwilliger algebras have been studied for many types of association schemes, including Schur rings \cite{Bastian2, Bastiansgp,Bastian}.

Tanaka \cite{Tanaka} defined an almost commutative Terwilliger algebra. One of Tanaka's results is that when a Terwilliger algebra is almost commutative, the corresponding association scheme can be written as a wreath product in a very particular way. 

In this paper, we characterize which Schur rings of a cyclic group result in an almost commutative Terwilliger algebra. Recall that if $G$ is a cyclic group of prime order, then any Schur ring over $G$ is an orbit Schur ring. If $G$ is cyclic of non-prime order, then any Schur ring over $G$ is either $(i)$ trivial, $(ii)$ an orbit Schur ring, $(iii)$ a direct product, or $(iv)$ a wedge product (see \cite{LeungI}).

% After studying orbit Schur rings of cyclic groups, we then find all almost commutative Terwilliger algebras for a Schur ring over a cyclic group. 

In this paper we let $\mathcal{C}_n=\langle z\rangle$ denote the cyclic group of order $n$ written multiplicatively. If $K$ is a characteristic subgroup of $G$, we write $K\text{ char }G$. We let $\mathcal{U}_n=\{z^u\colon \gcd(u,n)=1\}$ for any $n$ and let $\phi$ be the Euler-Totient function. For any group $G$ and $H\leq \Aut(G)$, we let $\S(G,H)$ be the orbit Schur ring over $G$ constructed using $H$. For $g\in G$ we let $O(g)$ denote the $H-$orbit of $g$. For $k\geq 1$ with $\gcd(k,n)=1$, we let $\varphi_k$ be the automorphism of $\mathcal{C}_n=\langle z\rangle$ defined by $\varphi_k(z)=z^k$. Given a Schur ring $\S$ over a group $G$ we let $T(G,\S)$ be the corresponding Terwilliger algebra with the identity of $G$ as the base point (any other choice of base point gives an isomorphic Terwilliger algebra). Our main result is:

\begin{Thm}\label{thm:main}
Let $\S$ be a Schur ring over $\mathcal{C}_n$. Then the Terwilliger algebra $T(\mathcal{C}_n,\S)$ is almost commutative if and only if we have one of:
\begin{enumerate}[label=(\roman*)]
    \item $\S$ is trivial;
    \item $\S$ is an orbit Schur ring $\S=\S(\mathcal{C}_n,H)$, $H\leq \Aut(\mathcal{C}_n)$ for the following possibilities:
    \begin{enumerate}
        \item $H$ is the trivial subgroup of $\Aut(\mathcal{C}_n)$. In this case, $\dim T(\mathcal{C}_n,\S)=n^2$.
        % \item $|G|$ is prime and $\S(G,H)$ is the trivial Schur ring.
        \item $n=p^k$, $p$ a prime, $k\geq1$, and $H=\Aut(\mathcal{C_n})$. If $p=2$, then $\dim T(\mathcal{C}_n,\S)=\frac{1}{2}(3k^2+3k+2)$. If $p$ is odd, then $\dim T(\mathcal{C}_n,\S)=\frac{1}{2}(3k^2+5k+2)$.
        \item $n=p^k$ for an odd prime $p$, $k>1$, with $H=\langle \varphi\rangle$ where $\varphi(z)=z^{1+p}$. In this case, $\dim T(\mathcal{C}_n,\S)=\frac{1}{2}(3k^2p^2-6k^2p+3k^2-kp^2+6kp-5k+2)$.
    \end{enumerate}
    \item $\mathcal{C}_n=\mathcal{C}_r\times \mathcal{C}_s$, $\S=\S_1\times \S_2$ where $\S_1$ and $\S_2$ are the group algebra over $\mathcal{C}_r$ and $\mathcal{C}_s$ respectively;
    \item $\S=\S_1\wedge_K \T$ for subgroups $K\leq H\leq \mathcal{C}_n$, $K\normal \mathcal{C}_n$, with $\S_1$ over $H$, $\T$ over $\mathcal{C}_n/K$, where $T(H,\S_1)$, $T(\mathcal{C}_n/K,\T)$ are almost commutative and a further condition (see Theorem \ref{thm:wedgeac}).
\end{enumerate}
\end{Thm}

% \begin{Thm}\label{thm:main}
%     Let $G=\langle z\rangle$ be a finite cyclic group, $H\leq \Aut(G)$, and $\S=\S(G,H)$. Then the Terwilliger algebra $T(G,\S)$ is almost commutative if and only if we have one of:
% \end{Thm}

% \begin{Thm}\label{thm:main2}
%     Let $G$ be a finite cyclic group and $\S$ be a Schur ring over $\S$. The the corresponding Terwilliger algebra $T(G,\S)$ is almost commutative if and only if we have one of:
%     \begin{enumerate}
%         \item $\S$ is the trivial Schur ring;
%         \item $\S$ is an orbit Schur ring for some $H\leq \Aut(G)$, where $G$ and $H$ are as in Theorem \ref{thm:main};
%         \item $\S$ is a direct product Schur ring $\S=\mathfrak{T}\times \mathfrak{U}$ in which both $\mathfrak{T}$ and $\mathfrak{U}$ are group rings;
%         \item $\S$ is a wedge product Schur ring $\S=\mathfrak{T}\wedge \mathfrak{U}$ in which the Terwilliger algebras corresponding to both $\mathfrak{T}$ and $\mathfrak{U}$ are almost commutative.
%     \end{enumerate}
% \end{Thm}

This paper is organized as follows: Section 2 will give preliminarily details about almost commutative Terwilliger algebras. Section 3 will then give a classification of almost commutative Terwilliger algebras for orbit Schur rings on groups of prime power order. Section 4 will consider orbit Schur rings for cyclic groups of non-prime power order. More information on the wreath product structure of the association scheme determined by $\S$ will be given in Theorem \ref{thm:wreathautop} and Theorem \ref{thm:wreath1p}. Section 5 will discuss wedge products and conclude the proof of Theorem \ref{thm:main}. 

\section{Terwilliger Algebras}

Let $\Omega$ be a finite set. Suppose $A_0,A_1,\cdots, A_d\in M_{|\Omega|}(\C)$ are $0,1$ matrices with rows and columns indexed by the elements of $\Omega$. Let $A^t$ denote the transpose of the matrix $A$. If additionally, 
\begin{enumerate}[leftmargin=*]
    \item $A_0=I_{|\Omega|}$;
    \item for all $i=1,2,\dots, d$, we have $A_i^t=A_j$ for some $j=1,2,\dots ,d$;
    \item for all $i,j\in \{1,2,\dots, d\}$ we have $A_iA_j=\sum_{k=0}^d p_{ij}^k A_k$;
    \item none of the $A_i$ is equal to $0_{|\Omega|}$ and $\sum_{i=0}^d A_i$ is the all $1$ matrix;
\end{enumerate}
then we say that $\mathcal{A}=(\Omega,\{A_i\}_{i=0}^d)$ is an \emph{association scheme}. We call $A_0,A_1,\cdots, A_d$ the \emph{adjacency matrices} of the association scheme and the constants $p_{ij}^k$ the \emph{intersection numbers}. If $A_iA_j=A_jA_i$ for all $i,j$, we have a \emph{commutative association scheme}. 

The \emph{$1-$class association scheme} for $|\Omega|=n$ has adjacency matrices $A_0=I_n$ and $A_1=J_n-I_n$, were $J_n$ is the all $1$ matrix of size $n$. We denote this association scheme by $\mathcal{K}_n$.

Let $G$ be a finite group and $R$ be a commutative ring with $1$. For $C\subseteq G$, let $\overline{C}=\sum_{g\in C} g\in R[G]$ and $C^*=\{g^{-1}\colon g\in G\}$. Let $\mathcal{P}$ be a partition of $G$. We say $\mathfrak{S}=\Span_R\{\overline{C}|C\in \mathcal{P}\}$ is a \emph{Schur ring} (see \cite{Ma89,Wielandt49}) if:
\begin{enumerate}[leftmargin=*]
    \item $\{e\} \in \mathcal{P};$
    \item if $C \in \mathcal{P}$, then $C^* \in \mathcal{P};$
    \item for all $C, D \in \mathcal{P}$,
$\overline{C}\cdot\overline{D} = \sum_{E\in \mathcal{P}} \lambda_{CDE}\overline{E}$.
\end{enumerate}

We call the sets in $\mathcal{P}$ the \emph{principal sets} of $\S$ and the constants $\lambda_{CDE}$ the \emph{structure constants}. We let $\mathcal{D}(\S)$ denote the set of principal sets of $\S$. Schur rings give rise to association schemes in the following way. We define $|G|\times |G|$ matrices $A_i$, $0\leq i\leq d$, by
\[(A_i)_{xy}=\left\{\begin{array}{cc}
    1 & \text{ if } x^{-1}y\in P_i  \\
    0 & \text{ otherwise.} 
\end{array}\right.\]
With this construction $(G,\{A_i\}_{i=0}^d)$ is an association scheme. If the Schur ring is commutative, then the association scheme is commutative. The structure constants $\lambda_{ijk}$ coming from $\overline{P_i}\cdot \overline{P_j}=\sum_{P_k\in \mathcal{P}}\lambda_{ijk}\overline{P_k}$ are the intersection numbers of the association scheme. Partitioning $G$ by the $H-$orbits for some $H\leq \Aut(G)$ produces a Schur ring $\S(G,H)$ called an \emph{orbit Schur ring}.

The \emph{group scheme} $\mathcal{G}(G)$ is the association scheme corresponding to $\S(G,\Inn(G))$.

We call $\mathfrak{A}=\Span_\C\{A_i\colon 0\leq i\leq d\}$ the \emph{Bose-Mesner algebra} of the association scheme. We let $E_0,E_1,\dots, E_d$ be the primitive idempotents of $\mathfrak{A}$. As $\mathfrak{A}$ is closed under \emph{Hadamard product}, denoted $\circ$, we have $E_i\circ E_j\in \mathfrak{A}$. Then there are $q_{ij}^k\in \C$, called the \emph{Krein parameters}, such that
\[E_i\circ E_j=\frac{1}{|\Omega|}\sum_{k=0}^d q_{ij}^kE_k.\]

Fix $x\in \Omega$ and define diagonal matrices $E_i^*(x)$ \emph{with base point $x$} by
\[(E_i^*(x))_{yy}=(A_i)_{xy}\]
We call $\mathfrak{A}^*(x)=\Span_\C\{E_i^*(x)\colon 0\leq i\leq d\}$ the \emph{dual Bose-Mesner Algebra of $\mathcal{A}$ with respect to $x$}.  

For the association scheme $\mathcal{A}$ and $x\in \Omega$ we define the \emph{Terwilliger algebra} with base point $x\in \Omega$, denoted $T_x=T_x(\mathcal{A})$, to be the subalgebra of $M_{|\Omega|}(\C)$ generated by $\mathfrak{A}$ and $\mathfrak{A}^*(x)$. For an association scheme coming from a Schur ring, we always take $x=e$.

Let $T_0(x)=\Span_\C \{E_i^*(x)A_jE_k^*(x)\colon 0\leq i,j,k\leq d\}$.
We note that $\dim T_0(x)=|\{(i,j,k)\colon p_{ij}^k\neq 0\}|$ \cite{Bannaiarticle}. A Terwilliger algebra is \emph{triply regular} if $T_x=T_0(x)$.

If $|\Omega|>1$, then the Terwilliger algebra with base point $x\in \Omega$, is non-commutative and semi-simple for all $x\in \Omega$ \cite{Terwilliger1} and so it has a Wedderburn decomposition. Given $\mathcal{A}=(\Omega,\{A_i\}_{i=0}^d)$, the Terwilliger algebra $T_x$ always has an irreducible ideal $V$ of dimension $(d+1)^2$, called the \emph{primary component} (see \cite{WreathProduct1}).

Closely related to the Wedderburn decomposition of the Terwilliger algebra are those $T_x-$modules that are irreducible when $T_x$ acts on $\C^{|\Omega|}$ by left multiplication. Then $\C^{|\Omega|}$ decomposes into a direct sum of irreducible $T_x-$modules. Letting $\hat{x}\in \C^{|\Omega|}$ have $1$ in position $x$ and $0$ elsewhere, it is shown in \cite{Terwilliger1} that $\mathfrak{A}\hat{x}$ is an irreducible $T_x-$module of dimension $d+1$. We call this irreducible module the \emph{primary module}.

A Terwilliger algebra $T_x$ is \emph{almost commutative} (AC) if every non-primary irreducible $T_x-$module is $1-$ dimensional. There is a classification of such Terwilliger algebras when $\mathcal{A}$ is commutative:

\begin{Thm}[Tanaka, \cite{Tanaka}]\label{thm:tanaka}  Let $\mathcal{A}=(\Omega,\{A_i\}_{i=0}^d)$ be a commutative association scheme. Let $T_x$ be the Terwilliger algebra of $\mathcal{A}$ for some $x\in \Omega$. The following are equivalent:
    \begin{enumerate}[leftmargin=*]
        \item Every non-primary irreducible $T_x-$module is $1-$dimensional for some $x\in \Omega$.
        \item Every non-primary irreducible $T_x-$module is $1-$dimensional for all $x\in \Omega$.
        \item The $p_{ij}^h$ satisfy: for all distinct $h,i$ there is exactly one $j$ such that $p_{ij}^h\neq 0$ $(0\leq h,i,j\leq d)$.
        \item The $q_{ij}^h$ satisfy: for all distinct $h,i$ there is exactly one $j$ such that $q_{ij}^h\neq 0$ $(0\leq h,i,j\leq d)$.
        \item $\mathcal{A}$ is a wreath product of association schemes $\mathcal{A}_1,\mathcal{A}_2,\cdots, \mathcal{A}_n$ where each $\mathcal{A}_i$ is either a $1-$class association scheme or the group scheme of an abelian group.
    \end{enumerate} 
Moreover, a Terwilliger algebra satisfying these equivalent conditions is triply regular.
\end{Thm}

The focus of this paper is to determine exactly when $T(G,\S)$ is AC for a Schur ring $\S$ over a cyclic group. Then we study the resulting Terwilliger algebra. We will need:

 \begin{Thm}\label{thm:sequiv}[Theorem 2.2, \cite{Bastian2}]
    Let $G$ be a finite group. Let $\S$ be any commutative Schur ring over $G$ with principal sets $P_i$, $0\leq i\leq d$. Then $T(G,\S)$ is AC if and only if $\S$ has the following property: for all $x,y\in G$ if $x\in P_i$, $y\in P_j$, $xy\in P_h$, and $P_i\neq P_j^*$ then $P_iP_j=P_h$. 
\end{Thm}

\section{Cyclic groups of order $p^n$}\label{sec:p^n}

In this section $G=\mathcal{C}_{p^n}=\langle z\rangle$, for prime $p$ and $n\geq 1$. If $H$ is trivial, then $\S(G,H)$ is the group algebra and so $T(G,\S)$ is AC. We now show that when $H=\Aut(G)$ this is also the case.

\begin{Thm}\label{thm:2nauto}
    Let $\S=\S(G,\Aut(G))$. Then $T(G,\S)$ is AC.
\end{Thm}

\begin{proof}
    We let $U=\{a\in \{1,2,\cdots, p^n-1\}\colon \gcd(a,p)=1\}$. Every element of $\Aut(G)$ is determined by where it maps $z$ and $\varphi\in \Aut(G)$ is of the form $\varphi(z)=z^t$, $t\in U$. Note that $\mathcal{U}_{p^n}=\{z^a\colon a\in U\}=O(z)$ and $O(z^a)=\mathcal{U}_{p^n}^a=\{u^{a}\colon u\in \mathcal{U}_{p^n}\}$ for $z^a\in G$. For $z^a\in G$, we write $z^a=z^{p^mb}$ for $m\geq 0$ and $b\in U$. Then $z^{p^mb}\in \mathcal{U}_{p^n}^{p^m}$, so $O(z^a)=O(z^{p^m})$. 
    
    % Now let $z^a\in O(z)$. Then there exists some $\varphi\in \Aut(G)$ such that $\varphi(z)=z^a$. Then $a\in \mathcal{U}_{p^n}$, so $O(z)=\mathcal{U}_{p^n}$.

    % We claim $O(z^a)=\mathcal{U}_{p^n}^a=\{u^{a}\colon u\in \mathcal{U}_{p^n}\}$ for all $z^a\in G$. To see this let $z^b\in O(z^a)$. Then there exists some $\varphi\in \Aut(G)$ such that $\varphi(z^a)=z^b$. Therefore, $\varphi(z)^a=z^b$. However, $\varphi(z)\in \mathcal{U}_{p^n}$. So $z^b=(\varphi(z))^a\in \mathcal{U}_{p^n}^a$. So $O(z^a)\subseteq \mathcal{U}_{p^n}^a$. Next let $z^c\in \mathcal{U}_{p^n}^a$. Then $z^c=z^{at}$ for some $t\in U$. We pick $\psi\in \Aut(G)$ such that $\psi(z)=z^t$. Then $\psi(z^a)=(\psi(z))^a=z^{at}=z^c$. Thus, $z^c\in O(z^a)$ and $O(z^a)=\mathcal{U}_{p^n}^a$.

    Note that $O(z^a)=O(z^a)^*$, since $\varphi(z)=z^{-1}$ is an automorphism. Now consider $O(z^a)=O(z^{p^m}),O(z^b)=O(z^{p^\ell})$ for $z^a,z^b\in G$ such that $O(z^a)\neq O(z^b)^*$. As $O(z^a)\neq O(z^b)$, $m\neq \ell$. We assume $m>\ell$ and claim that $O(z^{p^m})O(z^{p^\ell})=O(z^{p^m+p^\ell})$. First let $z^x\in O(z^{p^m+p^{\ell}})=\mathcal{U}_{p^n}^{p^m+p^\ell}$. Then $z^x=z^{(p^m+p^\ell)q}$ for some $q$, with $p\nmid q$, and $z^x=z^{p^mq+p^\ell q}\in \mathcal{U}_{p^n}^{p^m}\mathcal{U}_{p^n}^{p^\ell}=O(z^{p^m})O(z^{p^\ell})$. So $O(z^{p^m+p^\ell})\subseteq O(z^{p^m})O(z^{p^\ell})$. 
    
    To show the other direction we first note that $O(z^{p^m+p^\ell})=\mathcal{U}_{p^n}^{p^m+p^\ell}=\mathcal{U}_{p^n}^{p^\ell(p^{m-\ell}+1)}=\mathcal{U}_{p^n}^{p^\ell}$. 
    
    % Since $m>\ell$, $p\nmid (p^{m-\ell}+1)$. Hence, $\mathcal{U}_{p^n}^{(p^{m-\ell}+1)}=\{z^{(p^{m-\ell}+1)u}\colon u\in U\}\subseteq \mathcal{U}_{p^n}$. Notice for any $z^a,z^b\in \mathcal{U}_{p^n}$ that if $z^{(p^{m-\ell}+1)a}=z^{(p^{m-\ell}+1)b}$, then $z^a=z^b$. Therefore, $|\mathcal{U}_{p^n}^{(p^{m-\ell}+1)}|=|\mathcal{U}_{p^n}|$ and so $\mathcal{U}_{p^n}^{p^{m-\ell}+1}=\mathcal{U}_{p^n}$, so $\mathcal{U}_{p^n}^{p^\ell(p^{m-\ell}+1)}=\mathcal{U}_{p^n}^{p^\ell}$. Therefore, $O(z^{p^m+p^\ell})=\mathcal{U}_{p^n}^{p^\ell}$.

    Now let $z^w\in O(z^{p^m})O(z^{p^\ell})$. Then $z^w=z^{p^mr+p^\ell s}$ for $p\nmid rs$. Thus, $z^w=z^{p^mr+p^\ell s}=z^{p^\ell(p^{m-\ell}r+s)}\in \mathcal{U}_{p^n}^{p^\ell}=O(z^{p^m+p^\ell})$. Thus, $O(z^{p^m})O(z^{p^\ell})=O(z^{p^m+p^\ell})$. So for all $z^a,z^b\in G$ with $O(z^a)\neq O(z^b)^*$ we have $O(z^a)O(z^b)=O(z^{a+b})$. Thus, $T(G,\S)$ is AC by Theorem \ref{thm:sequiv}.
\end{proof}

We will need the following:

\begin{Thm}[Corollary 1, \cite{Deaconescu}]\label{thm:unitsum}
    Let $n\geq 2$ and let $U_n=\{i\colon 1\leq i<n\colon \gcd(i,n)=1\}\subseteq \Z/n\Z$.
    \begin{enumerate}[label=(\alph*)]
        \item If $n$ is odd, then $\{u+v\colon u,v\in U_n\}=\{0,1,\cdots, n-1\}$.
        \item If $n$ is even, then $\{u+v\colon u,v\in U_n\}=\{0, 2, 4, \cdots, n-2\}$.
    \end{enumerate}
\end{Thm}

\begin{Prop}\label{prop:autodim}
    Let $\S=\S(G,\Aut(G))$. If
    \begin{enumerate}[label=(\alph*)]
        \item $p=2$, then $\dim T(G,\S)=\frac{1}{2}(3n^2+3n+2)$;
        \item $p$ is odd, then $\dim T(G,\S)=\frac{1}{2}(3n^2+5n+2)$.
    \end{enumerate}
\end{Prop}

\begin{proof}
    From Theorem \ref{thm:2nauto}, $T(G,\S)$ is AC, so by Theorem \ref{thm:tanaka}, $T(G,\S)$ is triply regular and $\dim T(G,\S)=|\{(i,j,k)\colon p_{ij}^k\neq 0\}|$. We let $P_0,P_1,\cdots, P_d$ be the principal sets of $\S$. If $O(g)=P_i$, then $P_{i'}=O(g^{-1})$.

    We let $U=\{a\in \{1,2,\cdots, p^n-1\}\colon \gcd(a,p)=1\}$, $\mathcal{U}_{p^n}=\{z^a\colon a\in U\}$. The $H-$orbits are of the form $O(z^{p^m})=\mathcal{U}_{p^n}^{p^m}=\{u^{p^m}\colon u\in \mathcal{U}_{p^n}\}$ for some $m$. This gives $n+1$ orbits.

    Consider $P_iP_j$ where $j\neq i'$, $P_i=O(x)$ and $P_j=O(y)$, so $O(x)\neq O(y^{-1})$. Since $T(G,\S)$ is AC by Theorem \ref{thm:sequiv}, $P_iP_j=O(x)O(y)=O(xy)$. Hence, for all $i,j$ with $i\neq j'$, there is a unique $k$ such that $p_{ij}^k\neq 0$. We have $n+1$ choices for $i$ and $n$ choices of $j$, giving $n^2+n$ different nonzero $p_{ij}^k$.

    Now we consider $P_iP_{i'}$. If $P_i=\{e\}$, we have $P_iP_{i'}=\{e\}$, so $p_{ii'}^i\neq 0$. Now assume $P_i\neq \{e\}$. Say $P_i=O(z^{p^m})$ for $0\leq m\leq n-1$. Since the map $z\mapsto z^{-1}$ is an automorphism, we have $P_{i'}=P_i$. Thus,
    \[P_iP_{i'}=P_iP_i=O(z^{p^m})O(z^{p^m})=\{z^{ap^m+bp^m}\colon a,b\in U\}=\{z^{p^m(a+b)}\colon a,b\in U\}.\]
    We consider two cases based on if $p=2$ or not.
    
    {\bf Case 1: }$p=2$. By Theorem \ref{thm:unitsum}$(b)$, $O(z^{2^m})O(z^{2^m})=\{z^{2^ma}\colon a\in \{0,2,4,\cdots, 2^n-2\}\}$. We claim that $O(z^{2^m})O(z^{2^m})=\bigcup_{i=m+1}^n O(z^{2^i})$. First let $z^{2^ma}\in O(z^{2^m})O(z^{2^m})$. Then $a\in \{0,2,\cdots, 2^n-2\}$. If $a=0$, then $z^{2^ma}=e\in O(z^{2^n})\subseteq \bigcup_{i=m+1}^n O(z^{2^{i}})$. Now assume $a=2^rs$ for $1\leq r<n$ and odd $s\leq 2^n-1$. Note that $s\in U$ and $z^{2^ma}=z^{2^{m+r}s}$. If $m+r< n$, then $m+1\leq m+r<n$ and $z^{2^{m+r}s}\in O(z^{2^{m+r}})\subseteq \bigcup_{i=m+1}^n O(z^{2^i})$. If $m+r\geq n$, then $z^{2^{m+r}s}=e\in O(e)\subseteq \bigcup_{i=m+1}^n O(z^{2^{i}})$. Hence, $O(z^{2^m})O(z^{2^m})\subseteq \bigcup_{i=m+1}^n O(z^{2^i})$. 
    
    Next let $w\in \bigcup_{i=m+1}^n O(z^{2^i})$. Then $w\in O(z^{2^j})$ for $m+1\leq j\leq n$. Let $j=m+t$, $1\leq t\leq n-m$, so $w=z^{b2^j}=z^{b2^{m+t}}=z^{2^mb2^t}$ for $1\leq t\leq n-m$ and $b\in U$. Then $b2^t$ is even and $b2^t=c+k2^n$ for some $c\in \{0,2,\cdots, 2^n-2\}$ and $k\in \N$. Then $w=z^{2^mb2^t}=z^{2^m(c+k2^n)}=z^{2^mc}\in O(z^{2^m})O(z^{2^m})$. Therefore, $O(z^{2^m})O(z^{2^m})=\bigcup_{i=m+1}^n O(z^{2^i})$. Thus, there are $n-m$ different $H-$orbits so that $O(z^{2^i})\subseteq O(z^{2^m})O(z^{2^m})$. So for a fixed $0\leq m\leq n-1$ there are $n-m$ nonzero $p_{ii}^k$. Ranging over all $0\leq m\leq n-1$, gives $\sum_{m=0}^{n-1} (n-m)=\frac{n(n+1)}{2}$ nonzero $p_{ii}^k$. Adding on the one additional nonzero $p_{ii}^k$ when $m=n$ give us a total of $\frac{n(n+1)}{2}+1$ nonzero $p_{ii}^k$.   
    
    {\bf Case 2: }$p$ is odd. By Theorem \ref{thm:unitsum}$(a)$, $O(z^{p^m})O(z^{p^m})=\{z^{p^ma}\colon a\in \{0,1,2,\cdots, p^n-1\}\}$. As in Case 1, we have $O(z^{p^m})O(z^{p^m})=\bigcup_{i=m}^n O(z^{p^i})$. 
    
    % First let $z^{p^ma}\in O(z^{p^m})O(z^{p^m})$. Then $a\in \{0,1,\cdots, p^n-1\}$. If $a=0$, then $z^{p^ma}=e\in O(z^{p^n})\subseteq \bigcup_{i=m}^n O(z^{p^{i}})$. Now assume $a\neq 0$. We can write $a=p^rs$ for $0\leq r<n$ some $s\in U$. We have $z^{p^ma}=z^{p^{m+r}s}$. If $m+r< n$, then $m\leq m+r<n$ and $z^{p^{m+r}s}\in O(z^{p^{m+r}})\subseteq \bigcup_{i=m}^n O(z^{p^i})$. If $m+r\geq n$, then $z^{p^{m+r}s}=e$, as $z$ has order $p^n$. Also $e\in O(e)=O(z^{p^n})\subseteq \bigcup_{i=m}^n O(z^{p^{i}})$. Hence, $O(z^{p^m})O(z^{p^m})\subseteq \bigcup_{i=m}^n O(z^{p^i})$. Next let $w\in \bigcup_{i=m}^n O(z^{p^i})$. Then $w\in O(z^{p^j})$ for $m\leq j\leq n$. Note $j=m+t$ for some $0\leq t\leq n-m$. This implies $w=z^{bp^j}=z^{bp^{m+t}}=z^{p^mbp^t}$ for some $1\leq t\leq n-m$ and $b\in U$. If $bp^t>p^n-1$, then $bp^t\geq p^n$ and $z^{p^mbp^t}=e\in O(z^{p^m})O(z^{p^m})$. Now assume $bp^t\leq p^n-1$. Then $bp^t\in \{0,1,\cdots, p^n-1\}$ and we have $z^{p^mbp^t}\in \{z^{p^ma}\colon a\in \{0,1,2,\cdots, p^n-1\}\}=O(z^{p^m})O(z^{p^m})$. Either way $w\in O(z^{p^m})O(z^{p^m})$. Therefore, $O(z^{p^m})O(z^{p^m})=\bigcup_{i=m}^n O(z^{p^i})$. Thus, there are $n-m+1$ different $H-$orbits so that $O(z^{p^i})\subseteq O(z^{p^m})O(z^{p^m})$. This means for a fixed $0\leq m\leq n-1$ we have found there are $n-m+1$ nonzero $p_{ii}^k$. 
    
    \noindent Ranging over all $0\leq m\leq n-1$ we have found $\sum_{m=0}^{n-1} (n-m+1)=\frac{n(n+1)}{2}+n$ nonzero $p_{ii}^k$. Adding on the one additional nonzero $p_{ii}^k$ when $m=n$ gives us a total of $\frac{n(n+1)}{2}+n+1$ nonzero $p_{ii}^k$. So,
    \[\dim T(G,\S)=\left\{\begin{array}{cc}
        n^2+n+\frac{n(n+1)}{2}+1 & \text{if $p=2$}  \\
        \\
        n^2+n+\frac{n(n+1)}{2}+n+1 & \text{if $p\neq 2$}
    \end{array}\right.=\left\{\begin{array}{cc}
        \frac{1}{2}(3n^2+3n+2) & \text{if $p=2$}  \\
        \\
        \frac{1}{2}(3n^2+5n+2) & \text{if $p\neq 2$} 
    \end{array}\right.. \qedhere\]
\end{proof}

Let $H=\Aut(G)$ and $\S=\S(G,H)$. By Theorem \ref{thm:2nauto}, $T(G,\S)$ is almost commutative. This means we can write the association scheme corresponding to $\S$ as a wreath product. We show how to do this now. Let $\mathcal{B}(\mathcal{C}_{p^n})$ denote the association scheme corresponding to $\S$. We iteratively construct an ordering of $G$ as follows: We start with $e$. Then we add on the elements of $\mathcal{U}_{p^n}^{p^{n-1}}$, ordered in ascending order. Suppose $\Lambda_1$ is the sequence of elements we have ordered so far. Note that $\Lambda_1$ is an ordering of $\langle z^{p^{n-1}}\rangle$. We add on the elements $z^{p^{n-2}}\cdot\Lambda_1, z^{2p^{n-2}}\cdot \Lambda_1,\cdots, z^{(p-1)p^{n-2}}\cdot \Lambda_1$ in that order. Letting $\Lambda_2$ represent the ordering to this point (which consists of the elements of $\langle z^{p^{n-2}}\rangle$), we repeat this process with $z^{p^{n-3}}\cdot\Lambda_2, z^{2p^{n-3}}\cdot \Lambda_2,\cdots, z^{(p-1)p^{n-3}}\cdot \Lambda_2$. Call the order we have now $\Lambda_3$. Note that $\Lambda_3$ is an ordering of the elements of $\langle z^{p^{n-3}}\rangle$. We repeat this process so that the last elements we add are $z\cdot \Lambda_{n-1},z^2\cdot \Lambda_{n-1}\cdots, z^{(p-1)}\cdot \Lambda_{n-1}$.

\begin{Lem}\label{lem:autJ}
    Let $H=\Aut(G)$, $\S=\S(G,H)$, and let $A$ be the adjacency matrix for $O(z)$. Then $A=(J_p-I_p)\otimes J_{p^{n-1}}$. 
\end{Lem}

\begin{proof}
    Using the ordering given above, the proof is now similar to the proof of Lemma 4.4 in \cite{Bastiansgp}.
\end{proof}

\begin{Lem}\label{lem:autI}
    Let $H=\Aut(G)$ and $\S=\S(G,H)$. Let $A_i$ be the adjacency matrix of $O(z^{p^i})$ for $i\geq 1$. Then $A_i=I_p\otimes B$, where $B$ is the adjacency matrix corresponding to the $H-$orbit of $z^{p^{i-1}}$ in $\mathcal{B}(\langle z^p\rangle)$. 
\end{Lem}

\begin{proof}
    Using the ordering given above, the proof is similar to the proof of Lemma 4.3 in \cite{Bastiansgp}.
\end{proof}

\begin{Cor}\label{cor:wreathaut}
    We have $\mathcal{B}(\mathcal{C}_{p^n})=\mathcal{B}(\mathcal{C}_{p^{n-1}})\wr \mathcal{K}_p$.
\end{Cor}

\begin{proof}
    Let $B_0,B_1,\cdots, B_{n-1}$ be the adjacency matrices of $\mathcal{B}(\mathcal{C}_{p^{n-1}})$ where $B_i$ corresponds to the\\ $\Aut(\mathcal{C}_{p^{n-1}})-$orbit of $O(z^{p^i})$. Then the adjacency matrices of $\mathcal{B}(\mathcal{C}_{p^{n-1}})\wr \mathcal{K}_p$ are
    \[D_i=I_p\otimes B_i\ \text{(for $0\leq i\leq n-1$), } F_{n}=(J_p-I_p)\otimes J_{p^{n-1}}.\]
    Let $A_0,A_1,\cdots, A_n$ be the adjacency matrices of $\mathcal{B}(\mathcal{C}_{p^n})$, where $A_i$ corresponds to the $\Aut(G)-$orbit of $O(z^{p^i})$ in $\mathcal{C}_{p^n}$. By Lemma \ref{lem:autJ}, $A_0=(J_p-I_p)\otimes J_{p^{n-1}}$ as $A_0$ corresponds to the class $O(z)$. We note that $\langle z^p\rangle\cong \mathcal{C}_{p^{n-1}}$ and by Lemma \ref{lem:autI}, $A_i=I_p\otimes B_{i-1}$ for $1\leq i\leq n$. Thus, the adjacency matrices of $\mathcal{B}(\mathcal{C}_{p^n})$ are exactly the same as those of $\mathcal{B}(\mathcal{C}_{p^{n-1}})\wr \mathcal{K}_p$. Therefore, $\mathcal{B}(\mathcal{C}_{p^n})=\mathcal{B}(\mathcal{C}_{p^{n-1}})\wr \mathcal{K}_p$.
\end{proof}

\begin{Thm}\label{thm:wreathautop}
   Let $H=\Aut(G)$, and $\S=\S(G,H)$. Then \[\mathcal{B}(\mathcal{C}_{p^n})=\mathcal{K}_p\wr \mathcal{K}_p\wr\cdots \wr \mathcal{K}_p\] 
    where there are $n$ copies of $\mathcal{K}_p$ in this wreath product. 
\end{Thm}

\begin{proof}
    This result follows immediately by induction using Corollary \ref{cor:wreathaut}.
\end{proof}

% \begin{proof}
%     We proceed by induction on $n$. When $n=1$, we have $\mathfrak{B}(\mathcal{C}_p)=\mathcal{K}_p$ since the partition classes are $\{e\}$ and $G\setminus \{e\}$ in this case. Suppose we know the result is true for $\mathcal{C}_{p^m}$ for some $m\geq 1$. That is, $\mathcal{B}(\mathcal{C}_{p^{m}})=\mathcal{K}_p\wr \mathcal{K}_p\wr\cdots \mathcal{K}_p$ where there are $m$ copies of $\mathcal{K}_p$. We consider $\mathcal{C}_{p^{m+1}}$. By Corollary \ref{cor:wreathaut}, we have $\mathfrak{B}(\mathcal{C}_{p^{m+1}})=\mathcal{B}(\mathcal{C}_{p^m})\wr \mathcal{K}_p$. By the inductive hypothesis we have $\mathcal{B}(\mathcal{C}_{p^{m}})=\mathcal{K}_p\wr\mathcal{K}_p\wr\cdots \wr \mathcal{K}_p$. Therefore, $\mathfrak{B}(\mathcal{C}_{p^{m+1}})=\mathcal{K}_p\wr\mathcal{K}_p\wr\cdots \wr \mathcal{K}_p$ where there are $m+1$ copies of $\mathcal{K}_p$. This proves the inductive step.

%     Hence, by induction, we have for any $n\geq 1$ that $\mathcal{B}(\mathcal{C}_{p^n})=\mathcal{K}_p\wr \mathcal{K}_p\wr\cdots \wr \mathcal{K}_p$ where there are $n$ copies of $\mathcal{K}_p$ in this wreath product.
% \end{proof}

Now we shall consider $G=\mathcal{C}_{2^n}$. We need the following:

\begin{Prop}[5.7.13, \cite{Scott}]\label{thm:scott}
    Suppose $\mathcal{C}_{2^n}=\langle z\rangle$. If $n\geq 3$, then $\Aut(\mathcal{C}_{2^n})\cong H\times K$ where $H$ is the cyclic group of order $2^{n-2}$ generated by the map $\varphi_5$ and $K$ is the cyclic group of order $2$ generated by the map $\varphi_{-1}$.
\end{Prop}

\begin{Prop}\label{prop:proper2n}
    Let $\mathcal{C}_{2^n}=\langle z\rangle$ with $n\geq 3$. Let $H\leq \Aut(\mathcal{C}_{2^n})$. Then $H$ is one of the following:
    \begin{enumerate}
        \item $\langle \varphi_{5^k}\rangle$ for some $0\leq k< 2^{n-2}$;
        \item $\langle \varphi_{-5^k}\rangle$ for some $0\leq k< 2^{n-2}$;
        \item $\langle \varphi_{-1},\varphi_{5^k}\rangle$ for some $0\leq k<2^{n-2}$.
    \end{enumerate}
\end{Prop}

\begin{proof}
    This follows from Proposition \ref{thm:scott}.
\end{proof}

We now show that for $1< H< \Aut(G)$, if $\S=\S(G,H)$ then $T(G,\S)$ is not AC. We need:

% \begin{proof}
%     As $1+5^k\equiv 2(5)^{rk}\mod 2^n$, $1+5^k= 2(5)^{rk}+\ell2^n$ for some $\ell\in \Z$. Thus, $1+5^k\equiv 2(5)^{rk}+2\ell 2^{n-1}$. Hence, $1+5^k\equiv 2(5)^{rk} \mod 2^{n-1}$.
% \end{proof}

\begin{Lem}\label{lem:noncyc}
    Let $n\geq 3$ and $r\in \N$. If $5^k\not\equiv 1 \mod 2^n$, then $1+5^k\not\equiv 2(5)^{rk}\mod 2^n$. 
\end{Lem}

\begin{proof}
    We induct on $n\geq 3$. For $n=3$, we have $5^k\equiv 1 \mod 8$ if $k$ is even and $5^k\equiv 5 \mod 8$ if $k$ is odd. Then for all $k$ where $5^k\not\equiv 1 \mod 8$, we have $1+5^k\equiv 6 \mod 8$. However, $2(5)^{rk}\equiv 2 \mod 8$ for any $rk$. Thus for $n=3$, we have $1+5^k\not\equiv 2(5)^{rk}\mod 2^n$ if $5^k\not\equiv 1 \mod 2^n$.

    Now suppose $m\geq 3$ and that for all $k$ such that $5^k\not\equiv 1 \mod 2^m$, we have $1+5^k\not\equiv 2(5)^{rk} \mod 2^m$ for $r\in \N$. Fix $k$ such that $5^k\not\equiv 1 \mod 2^{m+1}$. We consider $2$ cases.\\
    {\bf Case 1: } $5^k\not\equiv 1 \mod 2^{m}$. Then $1+5^k\not\equiv 2(5)^{rk} \mod 2^m$ for $r\in \Z$. Then $1+5^k\not\equiv 2(5)^{rk} \mod 2^{m+1}$ for $r\in \N$.\\
    {\bf Case 2: } $5^k\equiv 1 \mod 2^m$. Then the order of $5$ in $\mathcal{C}_{2^m}^\times$ must divide $k$. The order of $5$ in $\mathcal{C}_{2^m}^\times$ is $2^{m-2}$. Hence $k=2^{m-2}\ell$, so $5^{rk}=5^{r2^{m-2}\ell}$ for $r\in \Z$. If $r=2t$ for $t\in \Z$, then $5^{rk}=5^{2t2^{m-2}\ell}=5^{2^{m-1}t\ell}\equiv 1 \mod 2^{m+1}$ as $5$ has order $2^{m-1}$ in $\mathcal{C}_{2^{m+1}}^*$. If $r=2t+1$ for $t\in \Z$, then $5^{rk}=5^{(2t+1)2^{m-2}\ell}=5^{2^{m-1}t\ell+2^{m-2}\ell}\equiv 5^{k}\mod 2^{m-1}$. Thus
    \[2(5^{rk})\equiv \left\{\begin{array}{cc}
        2 \mod 2^{m+1} & \text{ if $r$ is even}  \\
        2\cdot 5^{k} \mod 2^{m+1} & \text{ if $r$ is odd}
    \end{array}\right..\]
    For even $r$ if $1+5^k\equiv 2(5)^{rk} \mod 2^{m+1}$, then $1+5^k\equiv 2 \mod 2^{m+1}$. So $5^k\equiv 1 \mod 2^{m+1}$. If $r$ is odd with $1+5^k\equiv 2(5^{rk})\mod 2^{m+1}$, we have $1+5^k\equiv 2\cdot 5^{k}\mod 2^{m+1}$. Then $1\equiv 5^k \mod 2^{m+1}$. So whether $r$ is even or odd we get $5^k\equiv 1 \mod 2^{m+1}$, a contradiction.
\end{proof}

% We thus must have $1+5^k\not\equiv 2(5)^{rk} \mod 2^{m+1}$ for any choice of $r$ in this case.

    % Then by induction we have for all $n\geq 3$ and $k$ such that $5^k\not\equiv 1 \mod 2^n$, we have $1+5^k\not\equiv 2(5)^{rk}\mod 2^n$ for any $r\in \Z$.

\begin{Thm}\label{thm:2case'}
    Let $n\geq 3$, $G=\mathcal{C}_{2^n}$, $1<H<\Aut(G)$, and $\S=\S(G,H)$. Then $T(G,\S)$ is not AC.
\end{Thm}

\begin{proof}
     Suppose $T(G,\S)$ is AC. We consider three cases based on what $H$ can be from Proposition \ref{prop:proper2n}. 

    {\bf Case 1: }$H=\langle \varphi_{5^k}\rangle$, $0\leq k <2^{n-2}$. Since $H\neq \{1\}$, $5^k\not\equiv 1 \mod 2^n$. By Theorem \ref{thm:scott}, $\Aut(\mathcal{C}_{2^n})\cong \mathcal{C}_2\times \mathcal{C}_{2^{n-2}}$ where $\mathcal{C}_2=\langle \varphi_{-1}\rangle$ and $\mathcal{C}_{2^{n-2}}=\langle \varphi_5\rangle$. In, particular, $\varphi_{-1}\neq \varphi_{5^\ell}$ for all $\ell\in \Z$. Then for all $\psi\in H$ we have $\psi(z)\neq z^{-1}$, so $O(z)\neq O(z)^*$, and as $T(G,\S)$ is AC, $O(z)O(z)=O(z^2)$. Every element of $O(z^2)$ is of the form $z^{2(5^{rk})}$, $r\in \Z$. Now $z^{1+5^k}\in O(z)O(z)$, so $1+5^k\equiv 2(5)^{rk} \mod 2^{n}$. However, by Lemma \ref{lem:noncyc} we have $1+5^k\not\equiv 2(5^{rk})$, a contradiction.

    {\bf Case 2: }$H=\langle \varphi_{(-5)^k}\rangle$ for some $0\leq k<2^{n-2}$. If $k$ is even, then $(-5)^{rk}=5^{rk}$ for all $r$ and we have Case 1. So assume $k$ is odd. Since $H$ is nontrivial $(-5)^k\not\equiv 1 \mod 2^n$. As in Case $1$, we have $O(z)\neq O(z)^*$. As $T(G,\S)$ is AC we have $O(z)O(z)=O(z^2)$. Every element of $O(z^2)$ is of the form $z^{2(-5)^{rk}}$. Clearly $z^{1+(-5)^k}\in O(z)O(z)$, so $1+(-5)^k\equiv 2(-5)^{rk} \mod 2^n$. As $n\geq 3$ this implies, $1+(-5)^k\equiv 2(-5)^{rk} \mod 8$. We consider $2(-5)^{rk} \mod 8$. If $rk$ is even then $(-5)^{rk}\equiv 1 \mod 8$ and $2(-5)^{rk}\equiv 2 \mod 8$.  Then $1+(-5)^k\equiv 2\mod 8$, so $(-5)^k\equiv 1 \mod 8$. However, as $k$ is odd $(-5)^k\equiv -5 \mod 8$, a contradiction. If $rk$ is odd, then $(-5)^{rk}\equiv -5 \mod 8$, so $2(-5)^{rk}\equiv 6 \mod 8$. Then $1+(-5)^{rk}\equiv 6 \mod 8$, so $(-5)^k\equiv 5 \mod 8$. However, as $k$ is odd $(-5)^k\equiv -5 \mod 8$, a contradiction. So $T(G,\S)$ cannot be AC when $H=\langle \varphi_{(-5)^k}\rangle$.

    {\bf Case 3: }$H=\langle \varphi_{-1},\varphi_{5^k}\rangle$, $0\leq k<2^{n-2}$. Here $O(z)=O(z)^*$. Also every element of $O(z)$ is of the form $z^{\pm 5^{rk}}$, $r\in \Z$. Notice that if $z^5\in O(z)$ then $\varphi_5\in H$. However, as $\varphi_{-1},\varphi_5\in H$ generate $\Aut(G)$, $H=\Aut(G)$, a contradiction. So $z^5\not\in O(z)$. Then as $T(G,\S)$ is AC, $O(z)O(z^5)=O(z^6)$. Every element of $O(z^6)$ is of the form $z^{6(\pm 5^{rk})}$. Clearly $z^{5^{2k}-5}\in O(z)O(z^5)$ and so $5^{2k}-5\equiv 6(\pm 5^{rk})\mod 2^n$. As $n\geq 3$, this implies $5^{2k}-5\equiv 6(\pm 5^{rk})\mod 4$. However, $5^{2k}-5\equiv 1-1 \equiv 0 \mod 4$ and $6(\pm 5^{rk})\equiv 2 \mod 4$, a contradiction. So $T(G,\S)$ cannot be AC when $H=\langle \varphi_{-1},\varphi_{5^k}\rangle$.
\end{proof}

\begin{Thm}\label{thm:2case}
    Let $G=\mathcal{C}_{2^n}$, $H\leq \Aut(G)$, and $\S=\S(G,H)$. Then $T(G,\S)$ is AC if and only if $H$ is either $\Aut(G)$ or trivial.
\end{Thm}

\begin{proof}
    The case where $n=1$ is clear as the only Schur ring over $\mathcal{C}_2$ is $\C[\mathcal{C}_2]=\S(\mathcal{C}_2,\Aut(\mathcal{C}_2))$ and $T(G,\C[G])$ is AC. When $n=2$, $G=\mathcal{C}_4$. If $H$ is trivial $T(G,\S)$ is AC and if not then $H=\Aut(\mathcal{C}_4)$. By Theorem \ref{thm:2nauto}, $T(G,\S)$ is AC. For any $n\geq 3$, if $H$ is trivial then $T(G,\S)$ is AC as $\S$ is the group ring. By Theorem \ref{thm:2nauto}, if $H=\Aut(G)$, then $T(G,\S)$ is AC. For $1<H<\Aut(G)$ we have by Theorem \ref{thm:2case'} that $T(G,\S)$ is not AC. Hence, for all $n$, $T(G,\S)$ is AC if and only if $H$ is either $\Aut(G)$ or trivial.
\end{proof}

Before looking at $G=\mathcal{C}_{p^n}$ for $p$ odd, we prove the following lemma:

\begin{Lem}\label{lem:charac}
    Let $H\leq \Aut(G)$, and $\S=\S(G,H)$. Suppose that $T(G,\S)$ is AC. Now let $K\text{ char }G$ Then the $H-$orbits of $K$ form an orbit Schur ring over $K$, denoted by $\mathfrak{T}$ and $T(K,\mathfrak{T})$ is AC.
\end{Lem}

\begin{proof}
    As $K$ is characteristic, the orbits $C_0=\{e\},C_1,\cdots, C_d$ determine a Schur ring over $K$. Now assume that $C_i\neq C_j^*$ and consider $C_iC_j$. As $T(G,\S)$ is AC, $C_iC_j=D_t$ for some $H-$orbit $D_t$ of $G$ by Theorem \ref{thm:sequiv}. However, $D_t\subseteq K$, so $D_t=C_h$ for some $h\leq d$ and $C_iC_j=C_h$. So $T(K,\mathfrak{T})$ is AC by Theorem \ref{thm:sequiv}.
\end{proof}

% As $K$ is a characteristic subgroup of $G$, we have that for all $\varphi\in H$, $\varphi(K)=K$. Let $C_0,C_1,\cdots, C_d$ be the $H-$orbits of $G$ consisting of elements of $K$. Then $K=\bigcup_{i=0}^d C_i$. Clearly $\{e\}$ is one of the $C_i$ as $e\in K$. Also if $C_i\subseteq K$, then $C_i^*\subseteq K$. Hence, $C_i^*=C_j$ for some $j$. So the partition $C_0,\cdots, C_d$ of $K$ satisfies the first two conditions of being a Schur ring. Now consider $\overline{C_i}\cdot\overline{C_j}$. As this is a product of principal sets in $\S$ we have $\overline{C_i}\cdot\overline{C_j}=\sum_t \lambda_{ij}^t \overline{D_t}$ where the $D_t$ are principal classes in $\S$. However, $C_iC_j$ is a product of elements of $K$ and thus $C_iC_j\subseteq K$. So $D_t=C_h$ for some $H-$orbit of $K$ for each $t$. Then $\overline{C_i}\cdot\overline{C_j}=\sum_h \lambda_{ij}^h \overline{C_h}$ and $\mathfrak{T}$ is a Schur ring.

For the rest of this section we assume $p$ is odd and $G=\mathcal{C}_{p^n}=\langle z\rangle$ for $n\geq 1$. Then $\Aut(G)=\mathcal{C}_{\phi(p^n)}$. We start by proving there is $H\leq \Aut(G)$, where the corresponding Terwilliger algebra is AC.

\begin{Prop}\label{prop:1+p}
    Let $H=\langle \varphi_{1+p}\rangle$, and $\S=\S(G,H)$. Then $T(G,\S)$ is AC and $\dim T(G,\S)= \frac{1}{2}(3n^2p^2-6n^2p+3n^2-np^2+6np-5n+2)$. 
\end{Prop}

\begin{proof}
     For any $0\leq h\leq n$ we let $K_h=\langle z^{p^h}\rangle$, so 
    \[\{e\}=K_n < K_{n-1} < K_{n-2} < \cdots < K_{1} < K_0=G,\]
    and $|K_h|=p^{n-h}$ for all $h$. Clearly $\varphi_{1+p}(K_h)=K_h$. We now show that for $w\in G\setminus \{e\}$, if $K_i$ has the largest index such that $w\in K_i$, then $O(w)=wK_{i+1}$. As $w\in K_i\setminus K_{i+1}$, $w=z^{p^iu}$ for $u\in\Z$ where $p\nmid u$. Then $\varphi_{1+p}(w)=z^{(1+p)p^iu}=z^{p^iu+p^{i+1}u}=wz^{p^{i+1}u}\in wK_{i+1}$. We know $\varphi_{1+p}(K_{i+1})=K_{i+1}$, so $\varphi_{1+p}(wK_{i+1})=\varphi_{1+p}(w)\varphi_{1+p}(K_{i+1})\subseteq wK_{i+1}K_{i+1}=wK_{i+1}$. So $O(w)\subseteq wK_{i+1}$. Suppose $\varphi_{1+p}^s(w)=\varphi_{1+p}^r(w)$. Then $z^{(1+p)^sp^iu}=z^{(1+p)^rp^iu}$ and $(1+p)^sp^iu \equiv (1+p)^rp^iu\mod p^n$. As $u\in \mathcal{U}_{p^n}$, this gives $(1+p)^sp^i\equiv (1+p)^rp^i\mod p^n$. Hence, $(1+p)^s\equiv (1+p)^r \mod p^{n-i}$. Notice that $1+p$ has order $p^{n-i-1}$ in $\mathcal{C}_{p^{n-i}}^\times$. Hence, there are $p^{n-i-1}$ different values for $(1+p)^s \mod p^{n-i}$. Ranging over all of these values we have $p^{n-i-1}$ different elements in the $H-$orbit of $w$. Since this orbit is a subset of $wK_{i+1}$ and $|wK_{i+1}|=p^{n-(i+1)}$ we have $O(w)=wK_{i+1}$. 

   For $g\in G$, $O(e)O(g)=O(g)=O(g)O(e)$. Now let $z^x,z^y\in G\setminus \{e\}$. Suppose for $x,y\not\equiv 0 \mod p^n$, $O(z^x)\neq O(z^{-y})$. Let $i,j$ be the largest indices such that $z^x\in K_i$ and $z^y\in K_j$. Observe that $z^{x+y}\in K_{\min(i,j)}$. By the argument above $O(z^x)=z^xK_{i+1}$ and $O(z^y)=z^yK_{j+1}$. Since $K_h< K_m$ for $m<h$, $z^xK_{i+1} z^yK_{j+1}=z^{x+y}K_{\min(i+1,j+1)}=z^{x+y}K_{\min(i,j)+1}=O(z^{x+y})$ and $O(z^x)O(z^y)=O(z^{x+y})$. Hence, $T(G,\S)$ is AC by Theorem \ref{thm:sequiv}.    

   We now find $\dim(T(G,\S))$. By Theorem \ref{thm:tanaka}, $T(G,\S)$ is triply regular. Then $\dim T(G,\S)=|\{(i,j,k)\colon p_{ij}^k\neq 0\}|$. We let $P_0,P_1,\cdots, P_d$ be the principal sets of $\S$.

    First for $0\leq h\leq n$ we let $K_h=\langle z^{p^h}\rangle$. As we saw above, for all $w\in G\setminus \{e\}$ if $K_i$ has the largest index such that $w\in K_i$, then $O(w)=wK_{i+1}$ and $w\in K_i\setminus K_{i+1}$. Fix some $0\leq i\leq n-1$. Since $|K_i/K_{i+1}|=p$, there are $p$ cosets $wK_{i+1}$, $w\in K_i$. We have $p-1$ cosets of the form $wK_{i+1}$, $w\in K_i\setminus K_{i+1}$. Each of these is an $H-$orbit, so we have a total of $n(p-1)$ different $H-$orbits. Adding $O(e)=\{e\}$ gives $n(p-1)+1$ $H-$orbits.

    Consider $P_iP_j$ where $j\neq i'$. Suppose $P_i=O(x)$, $P_j=O(y)$, so that $O(x)\neq O(y^{-1})$. Since $T(G,\S)$ is AC we have by Theorem \ref{thm:sequiv}, $P_iP_j=O(x)O(y)=O(xy)$. Hence, for all $i,j$ with $i\neq j'$, there is a unique $k$ such that $p_{ij}^k\neq 0$. We have $n(p-1)+1$ choices for $i$ and $n(p-1)$ choices for $j\neq i'$, giving a total of $(np-n+1)(np-n)$ different nonzero $p_{ij}^k$.

    Now consider $P_iP_{i'}$. Clearly if $P_i=\{e\}$, we have $P_iP_{i'}=\{e\}$. This gives a single nonzero triple $(i,i',i)$. Next assume $P_i=O(x)\neq O(e)$. Then $P_iP_{i'}=O(x)O(x^{-1})$. Suppose $j$ is the largest index such that $x\in K_j$. Then $O(x)=xK_{j+1}$ and $O(x^{-1})=x^{-1}K_{j+1}$. Therefore, $P_iP_{i'}=xK_{j+1}x^{-1}K_{j+1}=K_{j+1}$. We now count the number of $H-$orbits of $G$ contained in $K_{j+1}$. To do so, we restrict to just considering $H-$orbits in $K_{j+1}$. We have
    \[\{e\}=K_n<K_{n-1}<K_{n-2}<\cdots K_{j+1}.\]
    For each $w\in K_{j+1}$ there is a largest index $h$ with $w\in K_{h}\setminus K_{h+1}$. Then $O(w)=wK_{h+1}$. As before we have $p-1$ cosets of the form $wK_{h+1}$ for each $j+1\leq h\leq n-1$. This gives a total of $(n-j-1)(p-1)+1$ $H-$orbits of $G$ in $K_{j+1}$, where the $1$ is added to account for $\{e\}$. Hence, we have $(n-j-1)(p-1)+1$ different $p_{ii'}^k\neq 0$ in this case.

    We note that for each $K_{j}$ there are $p-1$ different $H-$orbits $P_i$, such that $j$ is the largest index with $x\in K_j$. Then as we range over all $P_i\neq \{e\}$, for each $K_j$ there are $p-1$ of the $P_i$'s that have $(n-j-1)(p-1)+1$ nonzero $p_{ii'}^k$. This gives a total of $(p-1)\sum_{j=0}^{n-1} [(n-j-1)(p-1)+1]$, nonzero $p_{ii'}^k$. So in total we have  $(p-1)^2\frac{n(n-1)}{2}+(p-1)n+1$, triples $(i,i',k)$ such that $p_{ii'}^k\neq 0$.

    Combining we have
    \[\dim T(G,\S)=(np-n+1)(np-n)+(p-1)^2\frac{n(n-1)}{2}+(p-1)n+1\]
    \[=\frac{1}{2}(3n^2p^2-6n^2p+3n^2-np^2+6np-5n+2).\qedhere\]
\end{proof}

Let $H=\langle \varphi_{1+p} \rangle$ and assume $\S=\S(G,H)$. By Proposition \ref{prop:1+p}, $T(G,\S)$ is AC. We now show how to write the association scheme $\mathfrak{P}(\mathcal{C}_{p^n})$ corresponding to $\S$ as a wreath product. As in the proof of Proposition \ref{prop:1+p}, we set $K_h=\langle z^{p^h}\rangle$ for all $0\leq h\leq n$. We iteratively construct an ordering of $G$ as follows: We start with $\{e\}$. Then we add $z^{p^{n-1}},z^{2p^{n-1}},\cdots, z^{(p-1)p^{n-1}}$ in that order. These make up the elements of $K_{n-1}$. Then assuming we already have ordered all the elements of $K_{i+1}$, we add the elements in $z^{p^i}K_{i+1},z^{2p^i}K_{i+1},\cdots, z^{(p-1)p^i}K_{i+1}$ in that order (where the elements in $K_{i+1}$ are ordered according to the order we already had for them.) Repeating this process down to $z^{p}K_1, z^{2p}K_1,\cdots, z^{(p-1)p}K_1$, determines the order on $G$.

\begin{Prop}\label{prop:WreathJ1p}
    Let $H=\langle \varphi_{1+p} \rangle$ and $\S=\S(G,H)$. Let $P_i=z^{rp}K_1$ ($1\leq r\leq p-1$) be an $H-$orbit of $G$ outside $K_1$. Then $P_i=D\otimes J_{p^{n-1}}$, where $D$ is the adjacency matrix of $\{\varphi_{1+p}^r\}$ in $\mathcal{G}(\mathcal{C}_p)$ where $\mathcal{C}_p=\langle \varphi_{1+p}\rangle$. 
\end{Prop}

\begin{proof}
    Using the ordering given above, the proof is similar to the proof of Lemma 4.4 in \cite{Bastiansgp}.
\end{proof}

\begin{Prop}\label{prop:WreathI1p}
    Let $H=\langle \varphi_{1+p} \rangle$ and $\S=\S(G,H)$. Let $P_i$ be the $H-$orbit of $G$ in $K_1$ containing $w$. Then $P_i=I_{p}\otimes B$, where $B$ is the adjacency matrix corresponding to the $H-$orbit of $w$ in $\mathfrak{P}(K_{1})$. 
\end{Prop}

\begin{proof}
    Using the ordering given above, the proof is similar to the proof of Lemma 4.3 in \cite{Bastiansgp}.
\end{proof}

% \begin{proof}
%     We order $G$ as described above. We can split $A_i$ into blocks based on the cosets of $K_1$. For entry $x,y$ in the $z^{ap}K_1,z^{bp}K_1$ block to be nonzero we must have $x^{-1}y\in P_i$. Since, $x\in z^{ap}K_1$, and $y\in z^{bp}K_1$, we have $x^{-1}y\in z^{bp-ap}K_1$. However, $P_i\subseteq K_1$, so if $x^{-1}y\in P_i$ then $z^{bp-ap}K_1=K_1$. This implies $z^{bp}K_1=z^{ap}K_1$. Thus, the only blocks of $A_1$ that can be nonzero are the diagonal blocks. So every non-diagonal block of $A_i$ is the all $0$ matrix. 

%     Now we consider any diagonal block $z^{ap}K_1,z^{ap}K_1$. Entry $x,y$ of this block is nonzero in $A_i$ if and only if $x^{-1}y\in P_i$. This is true if and only if $x^{-1}y$ is in the $H-$orbit of $w$ in $K_1$. Since $K_1$ is a characteristic subgroup of $G$, the $H-$orbit of $w$ in $K_1$ is the same as its $H-$orbit in $G$. We know $\mathfrak{P}(K_1)$ is formed using the $H-$orbits of $K_1$. Then $x^{-1}y\in P_i$ is identical to the condition for the $x,y$ entry of $B$ to be nonzero. Hence, the $z^{ap}K_1,z^{ap}K_1$ block of $A_i$ must be $B$. As we chose this diagonal block of $A_i$ arbitrarily, we have that every $z^{ap}K_1,z^{ap}K_1$ block of $A_i$ is $B$. Hence, $A_i=I_p\otimes B$.
% \end{proof}

\begin{Cor}\label{cor:wreath1p}
    We have $\mathfrak{P}(\mathcal{C}_{p^n})=\mathfrak{P}(\mathcal{C}_{p^{n-1}})\wr \mathcal{G}(\mathcal{C}_p)$.
\end{Cor}

\begin{proof}
    If we let $B_0,B_1,\cdots, B_{(n-1)(p-1)+1}$ be the adjacency matrices of $\mathfrak{P}(\mathcal{C}_{p^{n-1}})$ and $D_0,D_1,\cdots, D_{p-1}$ be the adjacency matrices of $\mathcal{G}(\mathcal{C}_p)$, then the adjacency matrices of $\mathfrak{P}(\mathcal{C}_{p^{n-1}})\wr \mathcal{G}(\mathcal{C}_p)$ are
    \[F_i=I_p\otimes B_i\ (\text{for }0\leq i\leq (n-1)(p-1)+1),\ F_{d+j}=D_j\otimes J_{p^{n-1}}\ (\text{for }1\leq j\leq p-1).\]

    For any $H-$orbit $P_i\subseteq K_1$ of $G$, $P_i$ is an $H-$orbit of $K_1\cong \mathcal{C}_{p^{n-1}}$. By Proposition \ref{prop:WreathI1p}, $A_i=I_p\otimes B_i$ where $B_i$ is the adjacency matrix of $P_i$ in $\mathfrak{P}(\mathcal{C}_{p^{n-1}})$. From the counting of the $H-$orbits in the proof of Proposition \ref{prop:1+p}, there are $(n-1)(p-1)+1$ $H-$orbits of $G$ contained in $K_1$. This is how many adjacency matrices $\mathfrak{P}(\mathcal{C}_{p^{n-1}})$ has, so every adjacency matrix of $\mathfrak{P}(\mathcal{C}_{p^{n-1}})$ is used in forming the $A_i$.
    
    For any $H-$orbit $P_j$ of $G$ outside of $K_1$, we have $P_j=z^{rp}K_1$ and $A_j=D_t\otimes J_{p^{n-1}}$ where $D_t$ is the adjacency matrix for the class $\{r\}$ in $\mathcal{G}(\mathcal{C}_p)$. The $H-$orbits of $G$ outside of $K_1$ are $z^pK_1,z^{2p}K_1,\cdots, z^{(p-1)p}K_1$, so each adjacency matrix $D_t$ appears in exactly one of these Kronecker products.

    As every $H-$orbit of $G$ is either contained in $K_1$ or outside of $K_1$, as $K_1$ is characteristic, we have accounted for every $H-$orbit of $G$. Thus, the adjacency matrices of $\mathfrak{P}(G)$ are the same as those of $\mathfrak{P}(\mathcal{C}_{p^{n-1}})\wr \mathcal{G}(\mathcal{C}_p)$.
\end{proof}

\begin{Thm}\label{thm:wreath1p}
    Let $H=\langle \varphi_{1+p} \rangle$ and $\S=\S(G,H)$. Then \[\mathfrak{P}(\mathcal{C}_{p^n})=\mathcal{G}(\mathcal{C}_p)\wr \mathcal{G}(\mathcal{C}_p) \wr\cdots \wr \mathcal{G}(\mathcal{C}_p)\] 
    where there are $n$ copies of $\mathcal{G}(\mathcal{C}_p)$ in this wreath product. 
\end{Thm}

\begin{proof}
    This result follows immediately by induction using Corollary \ref{cor:wreath1p}.
\end{proof}

We now prove that for $1<H< \Aut(G)$, with $H\neq \langle \varphi_{1+p}\rangle$, and $\S=\S(G,H)$ that $T(G,\S)$ is not AC.

\begin{Lem}\label{lem:fixed}
    Suppose $1< H< \Aut(G)$. Suppose $\S=\S(G,H)$ and $T(G,\S)$ is AC. If $H=\langle \psi\rangle$ and $\psi$ fixes $z^a\neq e$, then $a=p^xr$ for some $r\in \Z$ and $\frac{n}{2}\leq x \leq n-1$. Furthermore, $z^{p^x}$ is also fixed.
\end{Lem}

\begin{proof}
Suppose $\psi(z)=z^k$ and let $z^a\neq 1$, where $\psi(z^a)=z^{ak}=z^a$. So $ak\equiv a \mod p^n$. Let $b,c\in \Z$ with $\gcd(p,b)=1$. If $z^{-b}\not\in O(z^{ca-b})$, then, as $T(G,\S)$ is AC, we have $O(z^b)O(z^{ca-b})=O(z^{ca})=\{z^{ca}\}$. Now $z^{kb+ca-b}\in O(z^b)O(z^{ca-b})$, so $z^{kb+ca-b}=z^{ca}$, and $kb\equiv b \mod p^n$. As $\gcd(b,p)=1$, we have $k\equiv 1 \mod p^n$, a contradiction. Therefore, $z^{-b}\in O(z^{ca-b})$ for all $b$ with $\gcd(p,b)=1$ and $c\in \Z$. 

In particular, $z^{-1}\in O(z^{a-1})$, $z^{-1}\in O(z^{2a-1})$, and so $O(z^{a-1})=O(z^{2a-1})$. As $z^{-1}\in O(z^{a-1})$, we have $z^{-k^\ell}=z^{a-1}$ for some $\ell\in \Z$. So $-k^\ell\equiv a-1 \mod p^n$, $0\leq \ell< |H|$. So $k^\ell\equiv 1-a \mod p^n$, and there is $0\leq w<|H|$ such that $z^{(a-1)k^w}=z^{2a-1}$ and $(a-1)k^w\equiv 2a-1 \mod p^n$. However, $ak\equiv a \mod p^n$, so $ak^w\equiv a \mod p^n$. Hence, $2a-1\equiv a-k^w\mod p^n$ and $k^w\equiv 1-a \equiv k^\ell \mod p^n$, so $k^{\ell-w}\equiv 1 \mod p^n$. Then $|H|\mid (\ell-w)$. So $\ell=w+h|H|$, $h\in \Z$, and $\ell=w$. So 
\[2a-1\equiv (a-1)k^\ell \equiv(a-1)(1-a)\equiv -a^2+2a-1 \mod p^n,\]
and $a^2\equiv 0 \mod p^n$. Therefore, $a=p^xr$, $\gcd(p,r)=1$ for some $r\in \Z$ and $\frac{n}{2}\leq x\leq n-1$. Note that $r\in\Z_{p^n}^\times$, so $z^{p^x}$ is also fixed by $\psi$.
\end{proof}

\begin{Thm}\label{thm:fixed}
    Suppose $1<H< \Aut(G)$, $\S=\S(G,H)$, and $T(G,\S)$ is AC. If $H=\langle \psi\rangle$ and $\psi$ fixes a non-identity element of $G$, then $H=\langle \varphi_{1+p}\rangle$.
\end{Thm}

\begin{proof}
Now $\psi(z)=z^k$ for some $k$. By Lemma \ref{lem:fixed}, $z^{p^x}$ is fixed for some $\frac{n}{2}\leq x\leq n-1$.

We induct on $n\geq 1$. If $n=1$, then $\frac{1}{2}\leq x\leq 0$. This is impossible, so there is no proper non-trivial subgroup of $\Aut(\mathcal{C}_p)$ that fixes a non-identity element and the result is vacuously true. If $n=2$, then $1\leq x\leq 1$, so $z^p$ is fixed by $\psi$. Then $pk\equiv p \mod p^2$, and $k-1\equiv 0 \mod p$. So $k=1+up$ for some $u\in \Z$. Note if $p\mid u$, then $k\equiv 1 \mod p^2$, which is false. So $p\nmid u$. Observe that $k^\ell \equiv 1+\ell u p \mod p^2$. Then $k^\ell \equiv 1 \mod p^2$ if and only if $p\mid \ell$, so the order of $k$ is $p$. Thus, $H$ is the unique subgroup of $\Aut(G)$ with order $p$. However, $\varphi_{1+p}(z)=z^{1+p}$ is also an element of order $p$. So $H=\langle \varphi_{1+p}\rangle$ and $|H|=p$.

Now assume that for $m\geq 2$, the only proper, nontrivial subgroup of $\Aut(\mathcal{C}_{p^m})$ such that $T(G,\S)$ is AC is generated by the map $\varphi_{1+p}$ and the subgroup has order $p^{m-1}$. We consider $1<H< \Aut(\mathcal{C}_{p^{m+1}})$, $H=\langle \psi\rangle$ with $\psi(z)=z^k$. Assume $\S=(\mathcal{C}_{p^{m+1}},H)$ and $T(\mathcal{C}_{p^{m+1}}, \S)$ is AC. 

Suppose $z^a$ is a non-identity element of $G$ of maximal order that is fixed by $\psi$. By Lemma \ref{lem:fixed}, $a=p^xr$ for some $r\in \Z$, $\gcd(p,r)=1$, and $\frac{m+1}{2}\leq x\leq m$. Furthermore, $z^{p^x}$ is fixed by $\psi$ for some $\frac{m+1}{2}\leq x\leq m$. Also $|z^{p^x}|=p^{m+1-x}$. 

Let $K$ be the subgroup of $\mathcal{C}_{p^{m+1}}$ of order $p^{m}$. Then $K$ is characteristic, so by Lemma \ref{lem:charac}, $T(K, \S_K)$ is AC, where $\S_K$ is the $H-$orbit Schur ring of $K$. We further note as $m+1-x<m$ since $x\geq \frac{m+1}{2}\geq \frac{2+1}{2}>1$, that $z^{p^x}\in K$. Also $z^{p^x}$ is not a generator of $K$, having order less than $m$. Therefore, when acting on $K$, $\psi$ fixes $z^x$. Furthermore, $\psi\neq {id}_k$ since the element of maximal order fixed by $\psi$ has order less than $p^m$. Since $\Aut(\mathcal{C}_{p^{m}})$ has no fixed points, $1<H<\Aut(\mathcal{C}_{p^m})$ and $T(K,\S_K)$ is AC. Then by induction, when $H$ is restricted to $\mathcal{C}_{p^m}$, it is generated by the map $y\mapsto y^{1+p}$ and has order $p^{m-1}$. Then as $\psi(z)=z^k$, we have $k\equiv(1+p)^u \mod p^m$ for some $u\in \Z$. So $k=(1+p)^u +vp^m$ for $u,v\in \Z$. Now $\psi$ is a generator of $H$ restricted to $K$. So $\psi=\varphi_{1+p}^u$ in $\mathcal{C}_{p^m}$. Since $H$ restricted to $\mathcal{C}_{p^m}$ has order $p^{m-1}$, $m\geq 2$ and $\psi$ generates it we must have $\gcd(u,p^{m-1})=1$. That is, $p\nmid u$. Suppose the order of $\psi$ is $\ell$. We have
\[k^\ell \equiv((1+p)^u+vp^m)^\ell\equiv \left(1+up+\binom{u}{2}p^2 +\cdots +p^u +vp^m\right)^\ell\mod p^{m+1}\]
\[\equiv 1+\ell u p+\ell p^2\left(\binom{u}{2}+\binom{\ell}{2}u^2 \right)+\cdots +\ell v p^m \mod p^{m+1}.\]
Then $k^\ell \equiv 1 \mod p^{m+1}$, and so $k^\ell-1\equiv 0 \mod p^{m+1}$. Thus,
\[\ell u p+\ell p^2\left(\binom{u}{2}+\binom{\ell}{2}u^2 \right)+\cdots +\ell v p^m \equiv 0 \mod p^{m+1}.\]
As $p$ divides every term in this sum, we have
\begin{equation}\label{eq:a.b}
    \ell u+\ell p\left(\binom{u}{2}+\binom{\ell}{2}u^2 \right)+\cdots +\ell v p^{m-1}\equiv 0 \mod p^{m}.
\end{equation}
Next $\psi\in \Aut(\mathcal{C}_{p^{m+1}})$ and as $\ell$ must divide $(p-1)p^m$ we have $\ell=bp^c$ for some $0\leq c\leq m$ and $b\mid p-1$. However, the order of $\psi$ when restricted to $\mathcal{C}_{p^m}$ is $p^{m-1}$, so $c\geq m-1$. Thus, Equation (\ref{eq:a.b}) implies
\[bu+b p\left(\binom{u}{2}+\binom{\ell}{2}u^2 \right)+\cdots +b v p^{m-1}\equiv 0 \mod p^{m-c}.\]
Thus,
\[u+p\left(\binom{u}{2}+\binom{\ell}{2}u^2 \right)+\cdots +v p^{m-1}\equiv 0 \mod p^{m-c}.\]
However, $m-c$ is either $0$ or $1$. If $m-c=1$, we have $u\equiv 0 \mod p$. This is false since $p\nmid u$. So $m-c=0$. This implies $\ell=bp^m$ for some $b\mid p-1$. However, $\ell$ is the order of $\psi$. Notice that
\[k^{p^m}\equiv 1+p^m u p+p^m p^2\left(\binom{u}{2}+\binom{p^m}{2}u^2 \right)+\cdots +p^m v p^m\equiv 1 \mod p^{m+1}.\]
Thus, the order of $\psi$ is $p^m$. Then $H$ is the unique subgroup of $\Aut(\mathcal{C}_{p^{m+1}})$ of order $p^m$. Note that $\varphi_{1+p}(z)=z^{1+p}$ has order $p^m$ in $\Aut(\mathcal{C}_{p^{m+1}})$ since $z^{1+p}$ has order $p^m$ in $(\mathcal{C}_{p^{m+1}})^*\cong \Aut(\mathcal{C}_{p^{m+1}})$, so $H=\langle \varphi_{1+p}\rangle$. This completes the inductive step. 
\end{proof}

\begin{Lem}\label{lem:Horder}
     Let $H\lneq \Aut(G)$, with $H=\langle \varphi_k\rangle$. Suppose that $\S=\S(G,H)$ and $T(G,\S)$ is AC. Then $|H|$ is a power of $p$.
\end{Lem}

\begin{proof}
    If $\mathcal{U}_{p^n}=O(z^{-1})$, then $|O(z^{-1})|=|\mathcal{U}_{p^n}|=p^{n-1}(p-1)=|\Aut(G)|$. However, $|O(z^{-1})|\leq |H|<|\Aut(G)|$. So there is $z^a\in \mathcal{U}_{p^n}\setminus O(z^{-1})$ and $O(z^a)\neq O(z^{-1})$. Then $O(z)O(z^a)=O(z^{1+a})$. Let
    \[R_1=\{z^{a+k^i}\colon 0\leq i<|H|\}; \quad R_2=\{z^{ka+k^i}\colon 0\leq i<|H|\}.\]
    Then $|R_1|=|R_2|=|H|$. Since $|O(z^{1+a})|\leq |H|$ and $R_1,R_2\subseteq O(z^{1+a})$, we have $R_1=R_2$. Multiplying the elements in $R_1$ and $R_2$ together we get $z^{\sum_{i=0}^{|H|-1}k^i+|H|a}=z^{\sum_{i=0}^{|H|-1}k^i + |H|(ka)}$. This implies, $|H|a\equiv |H|(ka) \mod p^n$, so that $|H|a(k-1)\equiv 0 \mod p^n$
    Thus, $|H|(k-1)\equiv 0 \mod p^n,$ and we write $|H|=p^rs$, where $s\mid (p-1)$, and $0\leq r\leq n-1$. If $\gcd(p,k-1)=1$, then $|H|(k-1)\not\equiv 0 \mod p^n$. Thus, $k=1+p^tu$, with $t\geq 1$, $\gcd(p,u)=1$, $t+r\geq n$, so $r\geq n-t\geq 1$.

    Since $k=1+p^tu$, $\gcd(p,u)=1$, $t\geq 1$, we have $k^b\equiv (1+p^tu)^b\equiv 1+bp^tu \mod p^{t+1}$ for any $b\geq 1$, so $k^p\equiv 1+p^{t+1}u \mod p^{t+1}$. So $k^p=1+p^{t+1}u_1$ for some $u_1\in \Z$. Now $k^{p^2}=(1+p^{t+1}u_1)^p\equiv 1+p^{t+2}u_1\mod p^{t+2}$ and by induction $k^{p^w}\equiv 1+p^{t+w}u_{w-1}\mod p^{t+w}$. Picking $w=n-t$ we find $k^{p^w}\equiv 1+p^{n}u_{n-t-1}\mod p^n$. Thus the order of $k$ in $\Z_{p^n}^\times$ must divide $p^w$. That is, $|H|\mid p^w$.
\end{proof}

\begin{Thm}\label{thm:primepower}
    Suppose $H\leq \Aut(G)$ and $\S=\S(G,H)$. Then $T(G,\S)$ is AC if and only if $H$ is one of the following subgroups of $\Aut(G)$
    \begin{itemize}
        \item The identity subgroup;
        \item $\langle \varphi_{1+p}\rangle$;
        \item $\Aut(G)$.
    \end{itemize}
\end{Thm}

\begin{proof}
    First, if $H=\{\varphi_1\}$, then $\S$ is the group ring and $T(G,\S)$ is AC. By Proposition \ref{prop:1+p}, if $H=\langle \varphi_{1+p}\rangle$ then $T(G,\S)$ is AC, and if $H=\Aut(G)$, then by Theorem \ref{thm:2nauto}, $T(G,\S)$ is AC.

    Now assume $T(G,\S)$ is AC and that $1<H<\Aut(G)$. Then by Lemma \ref{lem:Horder}, $|H|=p^r$, $0< r\leq n-1$, so $H$ is a subgroup of the unique subgroup $K=\langle \varphi_{1+p}\rangle$ of $\Aut(G)$ of order $p^{n-1}$ and $\varphi_k=\varphi_{1+p}^\ell$ for some $\ell\in \Z$. We note $\varphi_{1+p}(z^{p^{n-1}})=z^{p^{n-1}(1+p)}=z^{p^{n-1}}$. Hence, $\varphi_k$ fixes a non-identity element of $G$ and $T(G,\S)$ is AC. So by Theorem \ref{thm:fixed} $H=\langle \varphi_{1+p}\rangle$.
\end{proof}

\section{Cyclic Groups That Are Not $p-$Groups}

Here we show that for $1\lneq H\leq\Aut(\mathcal{C}_n)$, where $n$ is not a prime power, $T(G,\S(\mathcal{C}_n,H))$ is not AC.

\begin{Lem}\label{lem:bigclass}
Let $G=\mathcal{C}_n=\langle z\rangle$ where distinct primes $p$ and $q$ divide $n$. Let $H\leq \Aut(G)$ be nontrivial. Then there is some $H-$orbit $C\subseteq G\setminus \mathcal{U}_n$ such that $|C|>1$.
\end{Lem}

\begin{proof}
    Let $|G|=p_1^{a_1}p_2^{a_2}\cdots p_k^{a_k}$, $k\geq 2$, where $p_1,p_2,\cdots, p_k$ are distinct primes. Suppose the result is false. Then, the orbit of $z^{p_i^{a_i}}$ is $\{z^{p_i^{a_i}}\}$ for all $i$ as $z^{p_i^{a_i}}\in G\setminus \mathcal{U}_n$. Let $\varphi_t\in H\setminus \{id\}$. Then $\varphi_t(z^{p_i^{a_i}})=z^{tp_i^{a_i}}$. But $z^{p_i^{a_i}}$ is fixed by $\varphi_t$, so $p_i^{a_i}(t-1) \equiv 0 \mod n$ and $(t-1) \equiv 0 \mod \frac{n}{p_i^{a_i}}$, thus $p_2^{a_2}\cdots p_k^{a_k}\mid (t-1)$ and $p_1^{a_1}p_3^{a_3}\cdots p_{a_k}^{a_k}\mid (t-1)$. This implies, $n\mid (t-1)$, a contradiction.
\end{proof}

Let us begin by considering cyclic groups of even, non-prime power order.

\begin{Lem}\label{lem:2odd}
    Let $G=\mathcal{C}_{2m}=\langle z\rangle$ for $m\in \Z$ odd, $H\leq \Aut(G)$ be nontrivial, and $\S=\S(G,H)$. Then $T(G,\S)$ is not AC.
\end{Lem}

\begin{proof}
    Note that $O(z^m)=\{z^m\}$ as $z^m$ is the only element of order $2$. We claim $O(z)\neq O(z^{m-1})^*$. Suppose $O(z)^*=O(z^{m-1})$. Then $z^{-1}\in O(z^{m-1})$. Hence, there exists some $\varphi\in H$ such that $\varphi(z^{-1})=z^{m-1}$. Thus, $\varphi(z)=z^{m+1}$, and since $m$ is odd, $\gcd(m+1,2m)$ is even. Then $z\mapsto z^{m+1}$ is not an automorphism of $G$ in this case, a contradiction. Therefore, $O(z)^*\neq O(z^{m-1})$.
    
    Now $|H|>1$ and $|O(z)|=|H|$, so there is some $b=z^k\in O(z)$ such that $b\neq z$. Then
    \[\overline{O(z)}\cdot \overline{O(z^{m-1})}=(z+z^k+\cdots )(z^{m-1}+\cdots)=z^m+z^{m+k-1}+\cdots =\overline{O(z^m)}+\overline{O(z^{m+k-1})}+\cdots.\]
    Therefore, $T(G,\S)$ cannot be $AC$.
\end{proof}

\begin{Thm}\label{thm:evencase}
    Let $G$ be a cyclic group of non-prime power even order. Let $1<H\leq \Aut(G)$ and $\S=\S(G,H)$. Then $T(G,\S)$ is not AC.
\end{Thm}

\begin{proof}
    Suppose $|G|=2^ar$ where $r$ is odd. We induct on $a$. For $a=1$, Lemma \ref{lem:2odd} gives the result. Now suppose for any cyclic group $L$, $|L|=2^mr$, $m\geq 1$ we have the following: for every $1\lneq H\leq \Aut(G)$, if $\S=\S(L,H)$, then $T(L,\S)$ is not AC. We now assume $|G|=2^{m+1}r$. Let $1<H\leq \Aut(G)$ and $\S=\S(G,H)$ be the orbit Schur ring. Suppose $T(G,\S)$ is AC. 

    Let $K_2 \text{\ char\ } G$ of index $2$. As $T(G,\S)$ is AC, by Lemma \ref{lem:charac}, $T(K_2,\S_{K_2})$ is AC. However, by induction for $1<M\leq \Aut(K_2)$, the orbit Schur ring $\mathfrak{T}=\S(K_2,M)$ gives $T(K_2,\mathfrak{T})$ that is not AC. Therefore, as $T(K_2,\S_{K_2})$ is AC, we see that $H|_{K_2}$ is trivial. That is, $H$ fixes every element of $K_2$. 

    Let $G=\langle z\rangle$, so $K_2=\langle z^2\rangle$. Let $\varphi_t\in H$. Then $\varphi_t(z^2)=z^{2}$, so $t\equiv 1 \mod 2^{m}r$. Then either $t=1$, or $t=1+2^{m}r$. Therefore, $H=\{id, \varphi_{1+2^mr}\}$.

    We know $|O(z)|=|H|$ and so $O(z)=\{z, z^{1+2^mp_1^{b_1}p_2^{b_2}\cdots p_m^{b_m}}\}$. Clearly $O(z)\neq O(z)^*$. We have
    \[O(z)O(z)=(z+z^{1+2^mr})(z+z^{1+2^mr})=z^2+z^{2+2^{m}r}+z^{2+2^{m}r}+z^{(1+2^mr)(1+2^mpr)}\]
    \[=z^2+2z^{2+2^{m}r}+z^{1+2^{m+1}r+2^{2m}p^2b}=2z^2+2z^{2+2^{m}r}.\]
    This is a sum of at least two distinct classes since $z^2\in K_2$ and $\varphi_t$ fixes $K_2$. Hence, $\{z^2\}$ is a class and we have a contradiction.
\end{proof}

Now for cyclic groups of odd non-prime power order we have:

\begin{Prop}\label{prop:paqcase}
    Let $n=p^aq$ for distinct odd primes $p$ and $q$, with $a\geq 1$. Let $G=\mathcal{C}_n$ and $H\leq \Aut(G)$. Let $\S=\S(G,H)$. Then $T(G,\S)$ is AC if and only if $H$ is trivial.
\end{Prop}

\begin{proof}
    First assume $T(G,\S)$ is AC. Then $G=K\times L$ where $K=\langle x\rangle$, $|K|=p^a$ and $L=\langle y\rangle$, $|L|=q$. Both $K$ and $L$ are characteristic. As $T(G,\S)$ is AC, by Lemma \ref{lem:charac}, $T(K,\S|_K)$ and $T(L,\S|_L)$ are AC. Since $T(K,\S|_K)$ is AC we have, by Theorem \ref{thm:primepower}, one of:
    \begin{enumerate}
        \item $H|_K$ is trivial;
        \item $H|_K$ is generated by the map $x\mapsto x^{1+p}$;
        \item $H|_K=\Aut(K)$.
    \end{enumerate}
    Since $|L|=q$ and $T(L,\S|_L)$ is AC by Theorem \ref{thm:primepower} either $H|_L=\{id\}$ or $H|_L=\Aut(L)$. We consider cases based on these possible restrictions of $H$. 

    {\bf Case 1: } $H|_L=\{id\}$. As $L\text{ char } G$, for $0\leq i<q$, $O(y^i)=\{y^i\}$. We consider three subcases based on $H|_K$.\\
    {\bf Subcase 1.1: }$H|_K=\{id\}$. Then $H=\{id\}$. \\
    {\bf Subcase 1.2: }$H|_K=\langle \varphi_{1+p\rangle}$. If $a=1$, then $\varphi_{1+p}$ is just the identity map and we are in subcase 1.1. So assume $a>1$. Notice $x^{-1}\in K$ and $K \text{ char } G$, so $O(x),O(x^{-1})\subseteq K$. Also $x,x^{1+p}\in O(x)$ and $x^{-1},x^{-1-p}\in O(x^{-1})$. We consider $O(x^{-1})O(xy)$. As $O(x^{-1})\subseteq K$ and $xy\not\in K$, we have $O(xy)\neq O(x)$. Since $T(G,\S)$ is AC we have $O(x^{-1})O(xy)=O(y)$. However, $x^{-1-p}xy=x^{-p}y\in O(x^{-1})O(xy)$. As $a>1$, $x^{-p}y\neq y$. However, $x^{-p}y\in O(x^{-1})O(xy)=O(y)=\{y\}$, a contradiction.\\
    {\bf Subcase 1.3: }$H|_K=\Aut(K)$. As $H|_K=\Aut(K)$ and $K \text{ char } G$, we have $O(x)$ is the orbit of $x$ in $K$ under $\Aut(K)$. Thus, $O(x)=\{x^r\colon \gcd(r,p)=1\}$. Notice $x,x^{-1}\in O(x)=O(x)^*$. We consider $O(x)O(xy)$. Since $O(x)\subseteq K$ and $O(xy)\not\subseteq K$, $O(x)^*\neq O(xy)$. Since $T(G,\S)$ is AC, $O(x)O(xy)=O(x^2y)$. However, $x^{-1}xy=y\in O(x)O(xy)$, so $y\in O(x^2y)$ and $O(x^2y)=O(y)=\{y\}$. Hence, $x^2=e$. However, the order of $x$ is $p^a$ which is odd. We thus have a contradiction.

    {\bf Case 2: }$H|_L=\Aut(L)$. Since $L \text{ char } G$ and $H|_L=\Aut(L)$, we have $O(y)=\{y,y^2,\cdots, y^{q-1}\}$. Note $O(y)=O(y)^*\subseteq L$. Since $xy\not\in L$ we have $O(y)^*\neq O(xy)$. Then as $T(G,\S)$ is AC, $O(xy)O(y)=O(xy^2)$. However, $xyy^{q-1}=x\in O(xy)O(y)$ as well. Then $O(x)=O(xy)O(y)=O(xy^2)$. Since $x\in K \text{ char } G$, $O(x)\subseteq K$, so $xy^2\in K$ since $xy^2\in O(xy^2)=O(x)$. Then $y^2\in K$, a contradiction. 

    From these cases, if $T(G,\S)$ is AC, then $H|_K=\{id\}$ and $H|_L=\{id\}$, so $H$ is trivial. But then $\S$ is the full group ring and $T(G,\S)\cong M_n(\C)$ is AC.
\end{proof}

At this point we are ready to prove the result for non-prime power odd order cyclic groups.

\begin{Thm}\label{thm:oddcase}
    Let $G$ be a finite cyclic group of non-prime power odd order. Let $1<H\leq \Aut(G)$ and $\S=\S(G,H)$. Then $T(G,\S)$ is not AC.
\end{Thm}

\begin{proof}
    We induct on the number $k$ of prime divisors of $|G|$. If $k=2$, $|G|=pq$, for distinct odd primes $p$ and $q$ and by by Proposition \ref{prop:paqcase}, $T(G,\S)$ is not AC as $H\neq \{1\}$.

    Assume for all odd non-prime power cyclic groups $G$ whose order has at most $k$ prime factors, and $1<H\leq \Aut(G)$ that $T(G,\S)$ is not AC. Now suppose $|G|=m$ where $m$ has $k+1$ prime factors and let $1<H\leq \Aut(G)$. Let $U_m=\{g\in \mathcal{C}_m\colon \gcd(g,m)=1\}$. As $m$ is not a prime power it has at least two distinct prime divisors. Then by Lemma \ref{lem:bigclass}, there exists an $H-$orbit $C\subseteq G\setminus U_m$ such that $|C|>1$. As $C$ contains no units we have $\langle C\rangle \lneq G$. Let $K$ be a maximal subgroup of $G$ containing $\langle C\rangle$. Then $K$ has index $q$ for some odd prime $q$. Then $|K|$ has $k$ prime factors. Now either $|K|$ is a prime power, or it is not.

    {\bf Case 1: }Suppose $|K|=p^k$ for some odd prime $p$. Then $|G|=|K|\cdot |G\colon K|=p^kq$, and by Proposition \ref{prop:paqcase}, $T(G,\S)$ is not AC as $H$ is nontrivial.
    
    {\bf Case 2: }Suppose $|K|$ is a non-prime power and is odd. Then $K\text{ char }G$. We suppose $T(G,\S)$ is AC. By Lemma \ref{lem:charac}, $T(K,\S|_K)$ is AC. Notice though that $C\subseteq K$, is an $H-$orbit in $G$. It is then an $H-$orbit in $K$ as well since $K \text{ char } G$. Since $|C|>1$, $H|_K$ must be nontrivial. Then by the inductive hypothesis, $T(K,\S|_K)$ is not AC. This contradicts $T(K,\S|_K)$ being AC.
\end{proof}

\begin{Thm}\label{cor:allcase}
     Let $G$ be a finite cyclic group of non-prime power order. Let $H\leq \Aut(G)$. Let $\S=\S(G,H)$. Then $T(G,\S)$ is AC if and only if $H$ is trivial.
\end{Thm}

\begin{proof}
    First assume that $H$ is nontrivial. Then by Theorems \ref{thm:evencase} and \ref{thm:oddcase}, $T(G,\S)$ is not AC. If $H$ is trivial, then $\S$ is the full group ring and $T(G,\S)\cong M_{|G|}(\C)$ is AC.
\end{proof}

\section{AC Terwilliger Algebras over Cyclic Groups}

In this section we consider AC Terwilliger algebras for any Schur ring over $\mathcal{C}_n$, and prove Theorem \ref{thm:main}.

\begin{Thm}[Leung and Man, \cite{LeungII}]
Every Schur ring over a finite cyclic group is either a trivial Schur ring, orbit Schur ring, direct product Schur ring, or wedge product Schur ring. 
\end{Thm}

The case of a trivial Schur ring is the same as a trivial association scheme and has already been considered in \cite{Bastian, Tanaka}. The case of orbit Schur rings is covered in Theorems \ref{thm:2case}, \ref{thm:primepower}, and \ref{cor:allcase}. We next consider the direct product.

\begin{Thm}[Theorem 3.1, \cite{WreathProd}]\label{thm:direct}
    Let $\mathcal{A}=(X,R_i)$ and $\mathcal{B}=(Y,S_j)$ be association schemes. Let $x\in X$ and $y\in Y$. Let $T_x(\mathcal{A})$ be the Terwilliger algebra for $\mathcal{A}$ with base point $x$ and $T_y(\mathcal{B})$ be the Terwilliger algebra for $\mathcal{B}$ with base point $y$. Then
    \[T_{(x,y)}(\mathcal{A}\times \mathcal{B})\cong T_x(\mathcal{A})\otimes_\C T_y(\mathcal{B}).\]
\end{Thm}

\begin{Prop}\label{prop:dirnotac}
    Let $\mathcal{A}=(X,R_i)$ and $\mathcal{B}=(Y,S_j)$ be association schemes. Let $x\in X$ and $y\in Y$. Then $T_{(x,y)}(\mathcal{A}\times \mathcal{B})$ is AC if and only if $T_x(\mathcal{A})$ and $T_y(\mathcal{B})$ both only have the primary component.
\end{Prop}

\begin{proof}
    Let $V$ be the primary component of $T_x(\mathcal{A})$ and $U$ be the primary component of $T_y(\mathcal{B})$.

    If $T_x(\mathcal{A})=V$ and $T_y(\mathcal{B})=U$, then by Theorem \ref{thm:direct}, $T_{(x,y)}(\mathcal{A}\times \mathcal{B})\cong T_x(\mathcal{A})\otimes T_y(\mathcal{B})\cong V\otimes U$, which is the primary component. Then $T_{(x,y)}(\mathcal{A}\times \mathcal{B})$ is AC.

    Now suppose without loss of generality that $T(\mathcal{A})$ has a non-primary component $W$. Note $T_{(x,y)}(\mathcal{A}\times \mathcal{B})\cong T_x(\mathcal{A})\otimes_{\C} T_y(\mathcal{B})$ and the Wedderburn decomposition of a tensor product of two algebras over $\C$, is the tensor product over $\C$ of pairs of components from the two algebras, when using the direct product action. Then $W\otimes U$ is a component of $T_{(x,y)}(\mathcal{A}\times \mathcal{B})$ and is not the primary component. Also $\dim(W\otimes U)=\dim(W)\times \dim(U)\geq \dim(U)>1$. Thus, $T_{(x,y)}(\mathcal{A}\times \mathcal{B})$ has a non-primary component of dimension greater than $1$, so it is not AC.
\end{proof}

The following result of Tanaka can now be used to finish the direct product case.

\begin{Lem}[Lemma 8, \cite{Tanaka}]\label{lem:tan}
    Let $\mathcal{A}=(\Omega,\{A_i\})$ be a commutative association scheme. Let $W_0$ be the primary $T_x-$module. Then
        $\C^\Omega=W_0$ if and only if $\mathcal{A}$ is the group association scheme of a finite abelian group.
\end{Lem}

\begin{Thm}\label{cor:direct}
    For Schur rings $\S$ and $\T$ over groups $G$ and $H$ respectively we have that $T(G\times H, \S\times \T)$ is AC if and only if both $T(G,\S)$ and $T(H,\T)$ are the group algebra.
\end{Thm}

\begin{proof}
We note that the direct product of Schur rings corresponds to the direct product of the corresponding association schemes. As such Proposition \ref{prop:dirnotac} gives that for groups $G$ and $H$ with Schur rings $\S$ and $\mathfrak{T}$ over them respectively, we have $T(G\times H,\S\times \mathfrak{T})$ is AC if and only if both $T(G,\S)$ and $T(H,\mathfrak{T})$ only have the primary component. Then Lemma \ref{lem:tan}, implies $T(G,\S)$ and $T(H,\T)$ are the group algebra.    
\end{proof}

We now define the wedge product of Schur rings. Let $1<K\leq H\le G$ such that $K\trianglelefteq G$. Let $\S$ be a Schur ring over $H$ with $K$ as an $\S$-subgroup. Suppose $\pi : G\to G/K$ is the natural quotient map. Then $\pi(\S)$ is a Schur ring over $H/K$. Let $\T$ be a Schur ring over $G/K$ with $H/K$ a $\mathfrak{T}-$subgroup such that $\T_{H/K} = \pi(\S)$. Then we define the \emph{wedge product} $\S\wedge_K \T$ by the partition \[\D(\S\wedge_K \T) = \D(\S) \cup \{\pi^{-1}(D)\mid D\in \D(\T)\setminus\D(\T_{H/K})\}.\] Under these conditions, $\S\wedge_K \T$ is a Schur ring over $G$ (see \cite{LeungI} for details). Alternatively, a Schur ring $\S$ over $G$ is a wedge product if and only if there exist nontrivial, proper $\S$-subgroups $H,K\le G$ such that $K\le H$, $K\trianglelefteq G$, and every $\S$-class outside of $H$ is a union of $K$-cosets. 

We first consider when $T(G,\S\wedge_K \mathfrak{T})$ is AC. For the next several results we assume $1<K\leq H\leq G$ with $K\normal G$ and that $\pi\colon G\rightarrow G/K$ is the natural projection map. We have:

\begin{Prop}\label{prop:wedgeac1}
    Suppose $\S$ is a Schur ring over $H$ with $\overline{K}\in \S$, $\mathfrak{T}$ is a Schur ring over $G/K$ with $\overline{H/K}\in \mathfrak{T}$, and $\pi(\S)=\mathfrak{T}_{H/K}$. If $T(G,\S\wedge_K \mathfrak{T})$ is AC, then $T(H,\S)$ and $T(G/K,\mathfrak{T})$ are both AC.
\end{Prop}

\begin{proof}
    As $T(G,\S\wedge_K \mathfrak{T})$ is AC, we have by Theorem \ref{thm:sequiv} for $C,D\in \mathcal{D}(\S \wedge_K \T)$ with $C\neq D^*$, that $CD=E$ for some class $E$. In particular, as $\mathcal{D}(\S)\subseteq \mathcal{D}(\S\wedge_K \T)$ we have for $C,D\in \mathcal{D}(\S)$ with $C\neq D^*$, that $CD=E$ for some class $E$ and $T(H,\S)$ is AC.

    Next we let $A,B\in \mathcal{D}(\T)$ such that $A\neq B^*$. We consider four cases.
    
    {\bf Case 1: }$A,B\not\in \mathcal{D}(\T_{H/K})$. Then $\pi^{-1}(A),\pi^{-1}(B)\in \mathcal{D}(\S\wedge_K \T)$. If $\pi^{-1}(A)=\pi^{-1}(B)^*$, then given $g\in \pi^{-1}(A)$ we have $g^{-1}\in \pi^{-1}(B)$. So $gK\in A$ and $g^{-1}K\in B$, which implies $A=B^*$. As this is false, $\pi^{-1}(A)\neq\pi^{-1}(B)^*$, and $\pi^{-1}(A)\pi^{-1}(B)=C$ for some class $C\in \mathcal{D}(\S\wedge_K \T)$. If $C\subseteq H$, then $C\in \mathcal{D}(\S)$, so $\pi(C)\in \mathcal{D}(\T_{H/K})$. Then $\pi(C)=\pi(\pi^{-1}(A)\pi^{-1}(B))=AB$. If $C\not\subseteq H$, then $C=\pi^{-1}(D)$ for some class $D\in \mathcal{D}(\T)\setminus \mathcal{D}(\T_{H/K})$. Thus, we have $\pi^{-1}(D)=\pi^{-1}(A)\pi^{-1}(B)$. So 
    \[D=\pi(\pi^{-1}(D))=\pi(\pi^{-1}(A)\pi^{-1}(B))=AB.\]
    Hence, we see $AB$ is a single class in this case.

    {\bf Case 2: }$A\in \mathcal{D}(\T_{H/K})$ and $B\not\in \mathcal{D}(\T_{H/K})$. Then $A=\pi(C)$ for some $C\in \mathcal{D}(\S)$ and $\pi^{-1}(B)\in \mathcal{D}(\S\wedge_K \T)$. Since $C^*\subseteq H$ as $C\subseteq H$ we have by Theorem \ref{thm:sequiv}, $C\pi^{-1}(B)=D$ for some class $D\in \mathcal{D}(\S\wedge_K \T)$. If $D\subseteq H$, then $\pi(D)\in \mathcal{D}(\T_{H/K})$ and $\pi(D)=\pi(C\pi^{-1}(B))=\pi(C)B=AB$. If $D\not\subseteq H$, then $D=\pi^{-1}(E)$ for some class $E\in \mathcal{D}(\T)\setminus \mathcal{D}(\T_{H/K})$. We have 
    \[E=\pi(\pi^{-1}(E))=\pi(D)=\pi(C\pi^{-1}(B))=\pi(C)B=AB.\]
    Hence, $AB$ is a single class in this case.

    {\bf Case 3: }$A\not\in \mathcal{D}(\T_{H/K})$ and $B\in \mathcal{D}(\T_{H/K})$. The proof is similar to Case 2.

    {\bf Case 4: }$A,B\in \mathcal{D}(\T_{H/K})$. Then $A=\pi(C)$ and $B=\pi(D)$ for $C,D\in \mathcal{D}(\S)$. Suppose $C=D^*$. Then for $g\in C$ we have $g^{-1}\in D$ and $gK\in A$, $g^{-1}K\in B$. This contradicts $A\neq B^*$. Hence, $C\neq D^*$ and by Theorem \ref{thm:sequiv}, $CD=E$ for some class $E\in \mathcal{D}(\S\wedge_K \T)$. Clearly $E\in \mathcal{D}(\S)$, so $\pi(E)\in \mathcal{D}(\T_{H/K})$. We see $AB=\pi(C)\pi(D)=\pi(CD)=\pi(E)$. Hence, $AB$ is a single class in this case.

    Thus, in all cases $AB=C$ for a single class $C\in \mathcal{D}(\T)$ and by Theorem \ref{thm:sequiv}, $T(G/K,\T)$ is AC.
\end{proof}

In the following example we show that the converse of Proposition \ref{prop:wedgeac1} is false.

\begin{Example}
    Let $G=\mathcal{C}_{12}=\langle z\rangle$, $H=\langle z^2\rangle$, and $K=\langle z^4\rangle$. We take $\S$ to be the group ring on $H$ and $\T$ to be the group ring on $G/K$, both of which result in an AC Terwilliger algebra. Then
    \[\mathcal{D}(\S)=\{\{e\}, \{z^2\}, \{z^4\}, \{z^6\}, \{z^8\}, \{z^{10}\}\},\]
    \[\mathcal{D}(\T)=\{\{K\}, \{zK\}, \{z^2K\}, \{z^3K\}\},\]
    \[\mathcal{D}(\S\wedge_K \T)=\{\{e\}, \{z^2\}, \{z^4\}, \{z^6\}, \{z^8\}, \{z^{10}\}, \{z,z^5,z^9\}, \{z^3,z^7,z^{11}\}\}.\]
    Taking $\{z,z^5,z^{9}\}\{z,z^5,z^9\}=\{z^2\}\cup \{z^6\}\cup \{z^{10}\}$ we see that $T(G,\S\wedge_K \T)$ is not AC.
\end{Example}

\begin{Prop}\label{prop:wedgeac2}
    Suppose $\S$ is a Schur ring over $H$ with $\overline{K}\in \S$, $\mathfrak{T}$ is a Schur ring over $G/K$ with $\overline{H/K}\in \mathfrak{T}$, and $\pi(\S)=\mathfrak{T}_{H/K}$. If $T(G,\S\wedge_K \mathfrak{T})$ is AC, then for all classes $C,D\in \mathcal{D}(\T)\setminus \mathcal{D}(\T_{H/K})$, with $C\neq D^*$ such that $CD\in \mathcal{D}(\T_{H/K})$, we have $\pi^{-1}(CD)\in \mathcal{D}(\S\wedge_K \T)$.
\end{Prop}

\begin{proof}
    As $C,D\in \mathcal{D}(\T)\setminus \mathcal{D}(\T_{H/K})$ we have $\pi^{-1}(C),\pi^{-1}(D)\in \mathcal{D}(\S\wedge_K \T)$, say $A=\pi^{-1}(C)$ and $B=\pi^{-1}(D)$. Notice that
    \[\pi(AB)=\pi(A)\pi(B)=CD\subseteq H/K,\]
    so $AB\subseteq H$. Suppose $A=B^*$. Then for $a\in A$ we have $a^{-1}\in B$. So $\pi(a)\in C$, $\pi(a^{-1})\in D$, and $C=D^*$ a contradiction. Thus, $A\neq B^*$. As $T(G,\S\wedge_K \T)$ is AC we have $AB=E$ for some $E\in \mathcal{D}(\S)\subseteq \mathcal{D}(\S\wedge_K \T)$.

    We know $\pi(\S)=\T_{H/K}$. If $E\in \mathcal{D}(\S)$ is the only class such that $\pi(E)=CD$, then $\pi^{-1}(CD)=E$ and we are done. So assume $F\in \mathcal{D}(\S)$ with $E\neq F$ such that $\pi(F)=CD$. By Proposition \ref{prop:wedgeac1}, $T(H,\S)$ and $T(G/K,\T)$ are both AC, so as $E\neq F$, $EF^*=U$ for some class $U\in \mathcal{D}(\S)$. Also $CD=V$ for some class $V\in \mathcal{D}(\T)$ since $C\neq D^*$. Observe that
    \[\pi(U)=\pi(EF^*)=\pi(E)\pi(F)^*=VV^*.\]
    We know $U\in \mathcal{D}(\S)$, so $\pi(U)$ is a class in $\T_{H/K}$. Clearly $\{K\}\in VV^*$ and we know $\{K\}$ is a class in $\T_{H/K}$. Hence, $\pi(U)=\{K\}$ and $VV^*=\{K\}$, the identity class in $\T$. This is only possible if $V$ is a singleton, say $V=\{hK\}$ with $h\in H$. Thus, $CD=\{hK\}$ and $C$ and $D$ must be singletons as well, say $C=\{aK\}$ and $D=\{bK\}$. Since $CD=V=\{hK\}$, we can write $D=\{a^{-1}hK\}$.

    Observe that $A=\pi^{-1}(C)=aK$ and $B=\pi^{-1}(D)=a^{-1}hK$. Then $E=AB=hK$. However, \[F\subseteq \pi^{-1}(\pi(F))=\pi^{-1}(CD)=\pi^{-1}(\{hK\})=hK=E.\]
    Thus $F\subseteq E$, which contradicts $E,F\in \mathcal{D}(\S)$ with $E\neq F$, so $\pi^{-1}(CD)=E\in \mathcal{D}(\S\wedge_K \T)$.
\end{proof}

Now we consider the case when $T(H,\S)$ and $T(G/K,\T)$ are both AC.

\begin{Prop}\label{prop:SandTAC}
    Suppose $\S$ is a Schur ring over $H$ with $\overline{K}\in \S$, $\mathfrak{T}$ is a Schur ring over $G/K$ with $\overline{H/K}\in \mathfrak{T}$, and $\pi(\S)=\mathfrak{T}_{H/K}$. If $T(H,\S)$ and $T(G/K,\T)$ are AC and for all $C,D\in \mathcal{D}(\T)\setminus \mathcal{D}(\T_{H/K})$ with $C\neq D^*$, such that $CD\in \mathcal{D}(\T_{H/K})$ we have $\pi^{-1}(CD)\in\mathcal{D}(\S\wedge_K \T)$, then $T(G,\S\wedge_K T)$ is AC.
\end{Prop}

\begin{proof}
    Let $A,B\in \mathcal{D}(\S\wedge_K T)$ such that $A\neq B^*$. We consider three cases.
    
    {\bf Case 1: }$A,B\in \mathcal{D}(\S)$. Then $AB=C$ for some $C\in \mathcal{D}(\S)\subseteq \mathcal{D}(\S\wedge_K \T)$ as $T(H,\S)$ is AC.

    {\bf Case 2: }$A\in \mathcal{D}(\S)$ and $B=\pi^{-1}(C)$ for some $C\in \mathcal{D}(\T)\setminus \mathcal{D}(\T_{H/K})$. As $\pi(\S)=\T_{H/K}$, there exists some $D\in \mathcal{D}(\T_{H/K})$ such that $\pi(A)=D$. Since $D\in \mathcal{D}(\T_{H/K})$, $D^*\in \mathcal{D}(\T_{H/K})$ and $D^*\neq C$. Then $DC=E$ for some class $E\in \mathcal{D}(\T)$ as $T(G/K,\T)$ is AC. As $C\not\in \mathcal{D}(\T_{H/K})$, there is some $gK\in C$ such that $gK\not\in H/K$. Notice if $E\in \mathcal{D}(\T_{H/K})$, then for all $hK\in D$ we have $hKgK\in H/K$, which implies $gK\in H/K$. This is false, so $E\in \mathcal{D}(\T)\setminus \mathcal{D}(\T_{H/K})$. Thus, $\pi^{-1}(E)\in \mathcal{D}(\S\wedge_K \T)$. As $A\subseteq \pi^{-1}(\pi(A))=\pi^{-1}(D)$, we have $AB\subseteq \pi^{-1}(D)\pi^{-1}(C)\subseteq \pi^{-1}(DC)=\pi^{-1}(E)$. Thus, $AB\subseteq \pi^{-1}(E)$. However, $AB$ is a union of principal classes of $\S\wedge_K \T$. As it is contained inside the principal class $\pi^{-1}(E)$, we must have $AB=\pi^{-1}(E)$. 

    {\bf Case 3: }$A=\pi^{-1}(C)$ and $B=\pi^{-1}(D)$ for $C,D\in \mathcal{D}(\T)\setminus \mathcal{D}(\T_{H/K})$. If $C=D^*$, then for $gK\in C$ we have $g^{-1}K\in D$. Then $g\in A$ and $g^{-1}\in B$, so $A=B^*$. This is false, so $C\neq D^*$. Thus, $CD=E$ for some class $E\in \mathcal{D}(\T)\setminus \{K\}$, since $T(G/K,\T)$ is AC and $K\not\in CD$. If $E\in \mathcal{D}(\T)\setminus \mathcal{D}(\T_{H/K})$ then $\pi^{-1}(E)\in \mathcal{D}(\S\wedge_K \T)$. If $E\in \mathcal{D}(\T_{H/K})$ then by hypothesis $\pi^{-1}(E)\in \mathcal{D}(\S\wedge_K \T)$. So either way $\pi^{-1}(E)$ is a principal class of $\S\wedge_K \T$. Observe $AB=\pi^{-1}(C)\pi^{-1}(D)\subseteq \pi^{-1}(CD)=\pi^{-1}(E)$.
    As $AB$ is a union of principal classes of $\S\wedge_K \T$ and is contained inside the principal class $\pi^{-1}(E)$, we must have $AB=\pi^{-1}(E)$. 

    Thus, for all $A,B\in \mathcal{D}(\S\wedge_K \T)$ with $A\neq B^*$ we have $AB=C$ for some $C\in \mathcal{D}(\S\wedge_K \T)$. So $T(G,\S\wedge_K \T)$ is AC by Theorem \ref{thm:sequiv}.
\end{proof}

% We note there are cases where $E\in \mathcal{D}(\T_{H/K})\setminus \{K\}$ such that $\pi^{-1}(E)\not\in \mathcal{D}(\S\wedge_K \T)$ and yet $T(G,\S\wedge_K \T)$ is still AC, as seen by the following example. This explains the need for the added condition that $E=CD$ for $C,D\in \mathcal{D}(\T)\setminus \mathcal{D}(\T_{H/K})$.

% \begin{Example}\label{ex:SandTAC}
%     Let $G=\mathcal{C}_{27}=\langle z\rangle$, $H=\mathcal{C}_9$, and $K=\mathcal{C}_3$. We let $\S$ be the full group ring over $H$ and $\T$ be the Schur ring over $G/K$ generated by the map $z\mapsto z^{4}$. The classes of $\T$ are
%     \[P_0=\{K\},\ P_1=\{zK,z^4K,z^7K\},\ P_2=\{z^2K,z^5K, z^8K\},\ P_3=\{z^3K\},\ P_4=\{z^6K\}.\]
%     From Theorem \ref{thm:main}, we know $\S$ and $\T$ are AC. Furthermore, direct computation shows that $\S\wedge_K \T$ is AC. However, $\pi^{-1}(P_3),\pi^{-1}(P_4)\not\in \mathcal{D}(\S\wedge_K \T)$. 
% \end{Example}

We now give an equivalent condition to $T(G,\S\wedge_K \T)$ being AC. Using it, one can determine when a wedge product of two Schur rings results in an AC Terwilliger algebra.

\begin{Thm}\label{thm:wedgeac}
    Suppose $\S$ is a Schur ring over $H$ with $\overline{K}\in \S$, $\mathfrak{T}$ is a Schur ring over $G/K$ with $\overline{H/K}\in \mathfrak{T}$, and $\pi(\S)=\mathfrak{T}_{H/K}$. Then $T(G,\S\wedge_K \T)$ is AC if and only if 
    \begin{enumerate}
        \item $T(H,\S)$ and $T(G/K,\T)$ are both AC;
        \item For all classes $C,D\in \mathcal{D}(\T)\setminus \mathcal{D}(\T_{H/K})$, with $C\neq D^*$ such that $CD\in \mathcal{D}(\T_{H/K})$, we have $\pi^{-1}(CD)\in \mathcal{D}(\S\wedge_K \T)$. 
    \end{enumerate}
\end{Thm}

\begin{proof}
    If $T(G,\S\wedge_K \T)$ is AC then by Proposition \ref{prop:wedgeac1}, $T(H,\S)$ and $T(G/K,\T)$ are AC. By Proposition \ref{prop:wedgeac2}, we have for all classes $C,D\in \mathcal{D}(\T)\setminus \mathcal{D}(\T_{H/K})$, with $C\neq D^*$ such that $CD\in \mathcal{D}(\T_{H/K})$, we have $\pi^{-1}(CD)\in \mathcal{D}(\S\wedge_K \T)$.

    Conversely if $T(H,\S)$ and $T(G/K,\T)$ are AC with the property that for all classes $C,D\in \mathcal{D}(\T)\setminus \mathcal{D}(\T_{H/K})$, with $C\neq D^*$ such that $CD\in \mathcal{D}(\T_{H/K})$, we have $\pi^{-1}(CD)\in \mathcal{D}(\S\wedge_K \T)$, then by Proposition \ref{prop:SandTAC}, $T(G,\S\wedge_K \T)$ is AC.
\end{proof}

Theorem \ref{thm:main} now follows from Theorems \ref{thm:2case}, \ref{thm:primepower}, \ref{cor:allcase}, \ref{cor:direct}, \ref{thm:wedgeac}.

\printbibliography

\end{document}